\documentclass[smallextended]{svjour3}
\usepackage{graphicx}        % standard LaTeX graphics tool
\usepackage[bottom]{footmisc}% places footnotes at page bottom

\usepackage{amsmath,amssymb,mathtools}      
\journalname{Submitted}                    
\newcommand*{\id}{{\mathrm{id}}}

\newcommand*{\la}{{\langle}}                                                  
\newcommand*{\ra}{{\rangle}}                                                  
                                                        
\newcommand*{\R}{{\mathbb R}}                                                 
\newcommand*{\C}{{\mathbb{C}}}

\newcommand*{\Hb}{{\mathbb H}}                                                
                                                
\newcommand*{\Vb}{{\mathbb V}}
\newcommand*{\Wb}{{\mathbb W}}                                                
\newcommand*{\Sb}{{\mathbb S}}                                                
\newcommand*{\Nb}{{\mathbb N}}                                                
 
\newcommand*{\Fb}{{\mathbb F}}                   
                                                
\newcommand*{\Ib}{{\mathbb I}}

\newcommand*{\Jf}{{\mathfrak J}}  
\newcommand*{\af}{{\mathfrak a}}  
\newcommand*{\qf}{{\mathfrak q}}  
\newcommand*{\pf}{{\mathfrak p}} 
\newcommand*{\ef}{{\mathrm e}} 
\newcommand*{\pa}{{\partial}}

\newcommand*{\rd}{{\mathrm d}}

\newcommand*{\Gr}{{\mathrm{Gr}}} 
\newcommand*{\GL}{{\mathrm{GL}}}

\newcommand*{\Sbc}{{\Sb^{\mathrm{o}}\!\!}}

\begin{document}

\title{Grassmannian flows and applications to nonlinear partial differential equations}
\author{Margaret Beck \and Anastasia Doikou \and
Simon~J.A.~Malham \and Ioannis Stylianidis}
\titlerunning{Grassmannian flows and nonlinear PDEs}                                        
\authorrunning{Beck \and Doikou \and Malham \and Stylianidis}      

\institute{Margaret Beck \at Department of Mathematics and Statistics,
Boston University, Boston MA~02215, USA \email{mabeck@bu.edu}
\and Anastasia Doikou, Simon~J.A.~Malham and Ioannis Stylianidis
\at Maxwell Institute for Mathematical Sciences,        
and School of Mathematical and Computer Sciences,   
Heriot-Watt University, Edinburgh EH14 4AS, UK 
\email{A.Doikou@hw.ac.uk, S.J.A.Malham@hw.ac.uk, is11@hw.ac.uk}}
%\and 30th January 2018}
%
\date{30th January 2018}

\maketitle

\abstract{We show how solutions to a large class of partial differential equations 
with nonlocal Riccati-type nonlinearities can be generated from the corresponding 
linearized equations, from arbitrary initial data. It is well known that evolutionary 
matrix Riccati equations can be generated by projecting linear evolutionary flows on 
a Stiefel manifold onto a coordinate chart of the underlying Grassmann manifold.
Our method relies on extending this idea to the infinite dimensional case.
The key is an integral equation analogous to the Marchenko equation in integrable systems,
that represents the coodinate chart map. 
We show explicitly how to generate such solutions to scalar partial differential equations 
of arbitrary order with nonlocal quadratic nonlinearities using our approach. 
We provide numerical simulations that demonstrate the generation 
of solutions to Fisher--Kolmogorov--Petrovskii--Piskunov 
equations with nonlocal nonlinearities. 
We also indicate how the method might extend to more 
general classes of nonlinear partial differential systems.}

\section{Introduction}\label{sec:intro}
It is well known that solutions to many integrable nonlinear 
partial differential equations can be generated from solutions to
a linear integrable equation namely the Gel'fand--Levitan--Marchenko equation. 
It is an example of a generic dressing transformation which we 
shall express in the form 
\begin{equation}\label{eq:GLM}
g(x,y)=p(x,y)-\int_x^\infty g(x,z)q^\prime(z,y;x)\,\mathrm{d}z,
\end{equation}
for $y\geqslant x$. See Zakharov and Shabat~\cite{ZS} or 
Dodd, Eilbeck, Gibbon and Morris~\cite{DEGM} for more details.
Here all the functions shown may depend explicitly on time $t$, and we
suppose that $q^\prime$ and $p$ represent given data and $g$ is the solution.
Typically $p$ represents the scattering data and takes the form $p=p(x+y)$
while $q^\prime$ depends on $p$, for example $q^\prime=-p$ in the case of the 
Korteweg de Vries equation. See Ablowitz, Ramani and Segur~\cite{ARSII}
for more details. Typically given a nonlinear integrable partial differential equation,
the function $p$ is the solution to an associated linear system and the solution to 
the nonlinear integrable equation is given by $u=-2(\mathrm{d}/\mathrm{d}x)g(x,x)$.
See for example Drazin and Johnson~\cite[p.~86]{DJ} for the case of the 
Korteweg de Vries equation. The notion that the solution to a corresponding linear 
partial differential equation can be used to generate solutions to 
nonlinear integrable partial differential equations is addressed
in the review by Miura~\cite{Miura}. An explicit formula 
was provided by Dyson~\cite{Dyson} who showed that for the Korteweg de Vries equation
the solution to the Gel'fand--Levitan--Marchenko equation along the diagonal $g=g(x,x)$
can be expressed in terms of the derivative of the logarithm of a tau-function 
or Fredholm determinant. In a series of papers P\"oppe~\cite{P-SG,P-KdV,P-KP}, 
P\"oppe and Sattinger~\cite{PS-KP}, Bauhardt and P\"oppe~\cite{BP-ZS},
and Tracy and Widom~\cite{TW} expressed the solutions to further 
nonlinear integrable partial differential equations in terms of Fredholm determinants.  
Importantly P\"oppe~\cite{P-SG} explicitly states the idea that:
\begin{quote}
``For every soliton equation, there exists a \emph{linear} PDE (called a 
base equation) such that a map can be defined mapping a solution $p$ of the 
base equation to a solution $u$ of the soliton equation. The properties
of the soliton equation may be deduced from the corresponding properties of the
base equation which in turn are quite simple due to linearity. The map
$p\to u$ essentially consists of constructing a set of linear integral operators
using $p$ and computing their Fredholm determinants.''
\end{quote}

From our perspective, the solution $g$ to the dressing transformation
represents an element of a Fredholm Grassmann manifold, 
expressed in a given coordinate patch. Let us briefly explain
this perspective here. This will also
help motivate the structures we introduce herein. Our original
interest in Grassmann manifolds arose in spectral problems associated with
$n$th order linear operators on the real line which can 
be expressed in the form
\begin{subequations}\label{eq:spectralproblem}
\begin{align}
\pa_t\qf&=A\qf+B\pf\\
\pa_t\pf&=C\qf+D\pf,
\end{align}
\end{subequations}
where $\qf=\qf(t)\in\C^k$ and $\pf=\pf(t)\in\C^{n-k}$,
with natural numbers $1\leqslant k<n$. 
In these equations $A=A(t)\in\C^{k\times k}$, 
$B=B(t)\in\C^{k\times(n-k)}$, $C=C(t)\in\C^{(n-k)\times k}$ and 
$D=D(t)\in\C^{(n-k)\times(n-k)}$ are linear matrix operators. 
We assume as given that the matrix consisting of the blocks 
$A$, $B$, $C$ and $D$ has rank $n$ for all $t\in\R$. 
For example, in the case of elliptical eigenvalue problems,  
we have $(n-k)=k$ and $A=O$, $B=I_k$ and $C$ contains
the potential function. Then, with $t\in\R$ representing a spatial coordinate, 
the equations above are the corresponding
first order representation of such an eigenvalue problem and the first equation 
corresponds to simply setting the variable $\pf$ to be the spatial derivative 
of $\qf$. The goal is to solve such eigenvalue problems by shooting. 
In such an approach the far-field boundary conditions, let's focus on 
the left far-field for the moment, naturally determine a subspace of 
solutions which decay to zero exponentially fast, though in general 
at different exponential rates. The choice of $k$ above hitherto was
arbitrary, now we retrospectively choose it to be the dimension of
this subspace of solutions decaying exponentially in the left far-field.
We emphasize it is the data, in this case the far-field data, that
determines the dimension $k$ of the subspace we consider.
In principle we can integrate $k$ solutions from the left far-field 
forward in $t$ thus generating a continuous set of $k$-frames evolving
with $t\in\R$. If $(\qf_1,\pf_1)^{\mathrm{\tiny T}},\ldots,(\qf_k,\pf_k)^{\mathrm{\tiny T}}$
represent the solutions to the linear system~\eqref{eq:spectralproblem} above
that make up the components of the $k$-frame, we can represent them by
\begin{equation*}
\begin{pmatrix} Q \\ P\end{pmatrix}\coloneqq
\begin{pmatrix}
\qf_1 & \cdots &\qf_k \\
\pf_1 & \cdots &\pf_k 
\end{pmatrix},
\end{equation*}
where $Q\in\C^{k\times k}$ and $P\in\C^{(n-k)\times k}$. From the linear
system~\eqref{eq:spectralproblem} for $\qf$ and $\pf$ above, the matrices 
$Q=Q(t)$ and $P=P(t)$ naturally satisfy the linear matrix system
\begin{subequations}\label{eq:QandPevolution}
\begin{align}
\pa_t Q&=AQ+BP\\
\pa_t P&=CQ+DP.
\end{align}
\end{subequations}
To determine eigenvalues, by matching with the right far-field
boundary conditions, the minimal information required is that for the subspace only 
and not the complete frame information. The Grassmann manifold $\Gr(\C^n,\C^k)$ 
is the set of $k$-dimensional subpsaces in $\C^n$. It is thus the natural context 
for the subspace evolution and then matching. See Alexander, Gardner 
and Jones~\cite{AGJ90} for a comprehensive account; they used the Pl\"ucker
coordinate representation for the Grassmannian $\Gr(\C^n,\C^k)$.  
In Ledoux, Malham and Th\"ummler~\cite{LMT}, Ledoux, Malham, Niesen 
and Th\"ummler~\cite{LMNT} and Beck and Malham~\cite{BM} we chose
instead to directly project onto a coordinate patch of the 
Grassmannian $\Gr(\C^n,\C^k)$. Assuming that the matrix $Q\in\C^{k\times k}$
has rank $k$, we can achieve this as follows. We consider 
the transformation of coordinates here given by $Q^{-1}$ that renders the 
first $k<n$ coordinates as orthonormal thus generating the matrix
\begin{equation*}
\begin{pmatrix} I_k \\ G\end{pmatrix},
\end{equation*}
where $G(t)=P(t)\,Q^{-1}(t)$ for all $t\in\R$. Note this includes the data,
i.e.\/ $G(-\infty)=P(-\infty)Q^{-1}(-\infty)$.
The key point which underlies the ideas in this
paper is that $G=G(t)\in\C^{(n-k)\times k}$ evolves according to the
evolutionary Riccati equation
\begin{equation}\label{eq:Riccati}
\pa_t G=C+DG-G\,(A+BG),
\end{equation}
where $A$, $B$, $C$ and $D$ are the block matrices from the linear
evolutionary system~\eqref{eq:spectralproblem} above. This equation
for $G$ is straightforwardly derived by direct computation. 
The Grassmannian $\Gr(\C^n,\C^k)$ is a homogeneous manifold. As such
there are many numerical advantages to integrating along it, here this
corresponds to computing $G=G(t)$. For instance, the aforementioned differing 
far-field exponential growth rates are projected out and $G$ provides 
a useful succinct paramaterization of the subspace. It provides 
the natural extension of shooting methods to higher order 
linear spectral problems on the real line; also see Deng and Jones~\cite{DJ-M}. 
Let us call the procedure of deriving the Riccati equation~\eqref{eq:Riccati}
from the linear equations~\eqref{eq:QandPevolution} the 
\emph{forward} problem.

However in this finite dimensional context let us now turn this 
question around. Suppose our goal now is to solve the quadratically 
nonlinear evolution equation~\eqref{eq:Riccati} above for some
given data $G(-\infty)=G_0$. We assume of course the block
matrices $A$, $B$, $C$ and $D$ have the properties described above.
Let us call this the \emph{inverse} problem. With the forward problem 
described above in mind, given such a nonlinear evolution equation to solve,
we might naturally assume the nonlinear evolution equation~\eqref{eq:Riccati}
resulted from the projection of a linear Stiefel manifold flow onto
a coordinate patch of the underlying Grassmann manifold. From this
perspective, since all we are given is $G$ or indeed just the data $G_0$,
we can naturally assume $G_0$ was the result of such a Stiefel to Grassmann
manifold projection. In particular we are free to assume that the 
transformation $Q$ underlying the projection had rank $k$, and 
indeed $Q(-\infty)$, was simply $I_k$ itself. Thus if we suppose we were given
data $Q(-\infty)=I_k$ and $P(-\infty)=G_0$ and that $Q\in\C^{k\times k}$ 
and $P\in\C^{(n-k)\times k}$ satisfied the linear evolutionary
equations~\eqref{eq:QandPevolution} above then indeed $G=PQ^{-1}$
would solve the nonlinear evolution equation~\eqref{eq:Riccati} above.
Note that in this process there is nothing special about the data 
being prescribed at $t=-\infty$, it could be prescribed at any finite
value of $t$, for example $t=0$. In summary, suppose we want 
to solve the nonlinear Riccati equation~\eqref{eq:Riccati} for some
given data $G(0)=G_0\in\C^{(n-k)\times k}$. Then if we suppose the matrices 
$Q\in\C^{k\times k}$ and $P\in\C^{(n-k)\times k}$ satisfy the linear
system of equations~\eqref{eq:QandPevolution} with 
$Q(0)=I_k$ and $P(0)=G_0$, then the solution $G\in\C^{(n-k)\times k}$
to the linear relation $P=G\,Q$ solves the Riccati equation~\eqref{eq:Riccati}
on some possibly small but non-zero interval of existence. 

Our goal herein is to extend the idea just outlined 
to the infinite dimensional setting. Hereafter we always
think of $t\in[0,\infty)$ as an evolutionary time variable.
The natural extension of the finite rank (matrix) operator
setting above to the infinite dimensional case is to 
pass over to the corresponding setting with compact operators.  
Thus formally, now suppose $Q=Q(t)$ and $P=P(t)$ are linear
operators satisfying the linear system of 
evolution equations~\eqref{eq:QandPevolution} for $t\geqslant0$.
We assume that $A$ and $C$ are bounded operators, while the
operators $B$ and $D$ may be bounded or unbounded. 
We suppose the solution operators $Q=Q(t)$ and $P=P(t)$
are such that for some $T>0$, we have $Q(t)-\id$ and $P(t)$
are Hilbert--Schmidt operators for $t\in[0,T]$. Thus over
this time interval $Q(t)$ is a Fredholm operator. If the 
operators $B$ and $D$ are bounded then we require that 
$P$ lies in the subset of the class of Hilbert--Schmidt operators
characterized by their domains. In addition we now suppose that
$P=P(t)$ and $Q=Q(t)$ are related through a Hilbert--Schmidt operator
$G=G(t)$ as follows
\begin{equation}\label{eq:Riccatirelation}
P=G\,Q.
\end{equation}
We suppose herein this is a Fredholm equation for $G$ and  
not of Volterra type like the dressing transformation above.
We will return to this issue in our concluding section. 
As in the matrix case above, if we differentiate this Fredholm
relation with respect to time using the product rule, insert
the evolution equations~\eqref{eq:QandPevolution} for $Q=Q(t)$ and $P=P(t)$,
and then post-compose by $Q^{-1}$, then we obtain the Riccati evolution
equation~\eqref{eq:Riccati} for the Hilbert--Schmidt operator $G=G(t)$.
We emphasize that, as for the matrix case above, for some time interval
of existence $[0,T]$ with $T>0$, we can generate the solution to 
the Riccati equation~\eqref{eq:Riccati} with given initial data $G(0)=G_0$
by solving the two linear evolution equations~\eqref{eq:QandPevolution} 
with the initial data $Q(0)=\id$ and $P(0)=G_0$ and then solving the
third linear integral equation~\eqref{eq:Riccatirelation}. This is the
inverse problem in the infinite dimensional setting.

We now address how these operator equations are related
to evolutionary partial differential equations. Can we 
use the approach above to find solutions to 
evolutionary partial differential equations with nonlocal
quadratic nonlinearities in terms of solutions to the corresponding 
linearized evolutionary partial differential equations?
Suppose that $\Vb$ is a closed linear subspace of 
$\Hb\coloneqq L^2(\R;\R)\times L^2_{\mathrm{d}}(\R;\R)$. Here
$L^2_{\mathrm{d}}(\R;\R)\subseteq L^2(\R;\R)$ represents 
the subspace of $L^2(\R;\R)$ corresponding to the intersection 
of the domains of the operators $B$ and $D$.
Suppose further that we have the direct sum decomposition 
$\Hb=\Vb\oplus\Vb^\perp$, where $\Vb^\perp$ represents the 
closed subspace of $\Hb$ orthogonal to $\Vb$. 
As already intimated, suppose for some $T>0$ that for $t\in[0,T]$
we know: (i) $Q=Q(t)$ is a Fredholm operator from $\Vb$ to $\Vb$
of the form $Q=\id+Q^\prime$ where $Q^\prime=Q^\prime(t)$ is
a Hilbert--Schmidt operator on $\Vb$; and (ii) $P=P(t)$ is 
a Hilbert--Schmidt operator from $\Vb$ to $\Vb^\perp$. 
As such for $t\in[0,T]$ there exist integral kernels 
$q^\prime=q^\prime(x,y;t)$ and $p=p(x,y;t)$ with $x,y\in\R$
representing the action of the operators $Q'(t)$ and $P(t)$,
respectively. Let us define the following nonlocal product
for any two functions $g,g^\prime\in L^2(\R^2;\R)$ by
\begin{equation*}
\bigl(g\star g^\prime\bigr)(x,y)\coloneqq\int_\R g(x,z)\,g^\prime(z,y)\,\rd z.
\end{equation*}
Suppose now that the unbounded operators $B$ and $D$ are now
explicitly constant coefficient polynomial functions of $\pa_x$;
let us denote them by $b=b(\pa_x)$ and $d=d(\pa_x)$. Further 
suppose $A$ and $C$ are bounded Hilbert--Schmidt operators 
which can be represented via their integral kernels, say
$a=a(x,y;t)$ and $c=c(x,y;t)$, respectively. If $g=g(x,y;t)$
represents the integral kernel corresponding to the 
Hilbert--Schmidt operator $G=G(t)$ then we observe that
the two linear evolutionary equations~\eqref{eq:QandPevolution} 
and linear integral equation~\eqref{eq:Riccatirelation} can
be expressed as follows. We have 
\begin{enumerate}
\item \textit{Base equation:} $\pa_tp=c\star(\delta+q^\prime)+d\,p$;
\item \textit{Aux. equation:} $\pa_tq=a\star(\delta+q^\prime)+b\,p$;
\item \textit{Riccati relation:} $p=g\star(\delta+q^\prime)$.
\end{enumerate}
Here $\delta$ is the Dirac delta function representing the identity at the level
of integral kernels. The evolutionary equation for $g=g(x,y;t)$ corresponding to 
the Riccati evolution equation~\eqref{eq:Riccati} takes the form 
\begin{equation*}
\pa_tg=c+d\,g-g\star(a+b\,g).
\end{equation*}
This is an evolutionary partial differential equation for $g=g(x,y;t)$
with a nonlocal quadratic nonlinearity `$g\star(b\,g)$'.
Explicitly it has the form
\begin{equation*}
\pa_tg(x,y;t)=c(x,y;t)+d(\pa_x)\,g(x,y;t)
-\!\int_\R g(x,z;t)\bigl(a(z,y;t)+b(\pa_z)\,g(z,y;t)\bigr)\,\rd z.
\end{equation*}
The reason underlying the nomination of the base and auxiliary equations 
above is that in most of our examples we have $a\equiv c\equiv0$---let
us assume this for the sake of our present argument. We have outlined
the forward problem identified earlier at the partial differential equation
level. However our goal is to solve the inverse problem at this level:
given a nonlinear evolutionary partial differential equation of
the form above with some arbitrary initial data $g(x,y;0)=g_0(x,y)$,
can we re-engineer solutions to it from solutions to the corresponding 
base and auxiliary equations? The answer is yes. Given the 
nonlinear evolutionary partial differential equation $\pa_tg=d\,g-g\star(b\,g)$
with $b=b(\pa_x)$ and $d=d(\pa_x)$ as described above, suppose we solve
the corresponding linear base equation for $p=p(x,y;t)$, 
which is the linearized version of the given equation and 
consequently solve the auxiliary equation for $q^\prime=q^\prime(x,y;t)$.
Then solutions $g=g(x,y;t)$ to the nonlinear evolutionary 
partial differential equation are re-engineered/generated
by solving the Riccati relation for $p$ and $q^\prime$ 
which is a linear Fredholm integral equation. 

We explicitly demonstrate this procedure through two examples. 
We consider two Fisher--Kolmogorov--Petrovskii--Piskunov type equations
with nonlocal nonlinearities. The first has a nonlocal nonlinear term 
of the form `$g\star g$' where the product `$\star$' represents the special
case of convolution. The second has a nonlocal nonlinear term 
of the form `$g\star(b\,g)$' where $b=b(x)$ is a multiplicative function
corresponding to a correlation in the nonlinearity. In this latter case
the product `$\star$' has the general form as originally defined above.
In both these cases we show how solutions can be generated using the 
approach we propose from arbitrary initial data. We provide 
numerical simulations to confirm this. From these examples 
we also see how our procedure extends straightforwardly to any 
higher order diffusion. We additionally show how Burgers' equation 
and its solution using the corresponding base equation via 
the Cole--Hopf transformation fits into the context we have described here.

We emphasize that, as is well known for the 
Gel'fand--Levitan--Marchenko equation above which is of Volterra type, 
the procedure we have outlined works for most integrable systems, 
as demonstrated in Ablowitz, Ramani and Segur~\cite{ARSII} who assume 
$p=p(x+y)$ is a Hankel kernel. For example, we can generate solutions 
to the Korteweg de Vries equation from the Gel'fand--Levitan--Marchenko 
equation by setting $q^\prime=-p$. As another example, we can generate 
solutions to the nonlinear Schr\"odinger equation
by assuming $q^\prime(z,y;x)=\pm\int_x^{\infty}\overline{p}(z,\zeta)\,p(\zeta,y)\,\rd\zeta$
where $\overline{p}$ represents the complex conjugate of $p$. In this case it is
also well known that such solutions can be generated from a 2$\times$2 matrix-valued
dressing transformation. See  Zakharov and Shabat~\cite{ZS},  
Dodd, Eilbeck, Gibbon and Morris~\cite{DEGM} or Drazin and Johnson~\cite{DJ}
for more details. Further, the connection between integrable systems 
and infinite dimensional Grassmann manifolds was first made by Sato~\cite{SatoI,SatoII}.
Lastly Riccati systems are a central feature of optimal control systems.
The solution to a matrix Riccati equation provides the optimal continuous
feedback in optimal linear-quadratic control theory. See for example 
Bittanti, Laub and Willems~\cite{BLW}, Brockett and Byrnes~\cite{BB},
Hermann~\cite{HpartA,HpartB}, Hermann and Martin~\cite{HM},
Martin and Hermann~\cite{MH} and Zelikin~\cite{Z} for more details.
A comprehensive list of the related control literature can also be found
in Ledoux, Malham and Th\"ummler~\cite{LMT}.

Our paper is structured as follows. In Section~\ref{sec:Grassmannians} 
we define and outline the Grassmann manifolds in finite and infinite
dimensions that we require to give the appropriate context to our procedure.
Then in Section~\ref{sec:flows} we 
show how linear subspace flows induce Riccati flows in coordinate patches of
the corresponding Fredholm Grassmannian. We derive the equation for the 
evolution of the integral kernel associated with the Riccati flow. 
We then consider two pertinent examples in Section~\ref{sec:examples}.
Their solutions can be derived by solving the linear base 
and auxiliary partial differential equations (the subspace flow) 
and then solving the linear Fredholm equation representing 
the projection of the subspace flow onto a coordinate patch of the 
Fredholm Grassmannian. Then finally in Section~\ref{sec:conclu} we 
discuss possible extensions of our approach to other 
nonlinear partial differential equations.

\section{Grassmann manifolds}\label{sec:Grassmannians} 
Grassmann manifolds underlie the structure, development and solution
of the differential equations we consider herein. Hence we
introduce them here first in the finite dimensional, and then second
in the infinite dimensional, setting. There are many perspectives and prescriptions,
we choose the prescriptive path that takes us most efficiently to the infinite
dimensional setting we require herein.

Suppose we have a finite dimensional vector space say $\Hb=\C^n$ of dimension $n\in\Nb$.
Given an integer $k$ with $1\leqslant k<n$, the Grassmann manifold $\Gr(\C^n,\C^k)$ is 
defined to be the set of $k$-dimensional linear subspaces of $\C^n$. 
Let $\{\mathrm{e}_j\}_{j\in\{1,\ldots,n\}}$ denote the \emph{canonical basis} for $\C^n$,
where $\mathrm{e}_j$ is the $\C^n$-valued vector with one in the $j$th entry and
zeros in all the other entries. Suppose we are given a set of $k$ linearly 
independent vectors in $\C^n$ and we record them in the following $n\times k$ matrix:
\begin{equation*}
W=\begin{pmatrix}
w_{1,1} & \cdots & w_{1,k}\\
\vdots &        & \vdots \\
w_{n,1} & \cdots & w_{n,k}
\end{pmatrix}.
\end{equation*}
Each column is one of the linear independent vectors in $\C^n$. 
This matrix has rank $k$.
Naturally the columns of $W$ span a $k$-dimensional subspace or 
$k$-plane $\Wb$ in $\C^n$. Let us denote by $\Vb_0$ the 
\emph{canonical subspace} given by 
$\text{span}\{\mathrm{e}_1,\ldots,\mathrm{e}_k\}$, i.e.\/ 
the subspace prescribed by the first $k$ canonical basis
vectors which has the representation
\begin{equation*}
W_0\coloneqq\begin{pmatrix}
I_k\\
O
\end{pmatrix}.
\end{equation*}
Here $I_k$ is the $k\times k$ identity matrix. 
The span of the vectors $\{\mathrm{e}_{k+1},\ldots,\mathrm{e}_n\}$
represents the subspace $\Vb_0^\perp$, the $(n-k)$-dimensional subspace
of $\C^n$ orthogonal to $\Vb_0$. Suppose we are able to
project $\Wb$ onto $\Vb_0$. Then the projections 
$\mathrm{pr}\colon\Wb\to\Vb_0$ and $\mathrm{pr}\colon\Wb\to\Vb_0^\perp$ 
respectively give 
\begin{equation*}
W^\parallel=\begin{pmatrix}
w_{1,1} & \cdots & w_{1,k}\\
\vdots &        & \vdots \\
w_{k,1} & \cdots & w_{k,k}\\
0      & \cdots & 0\\
\vdots &        & \vdots \\
0      & \cdots & 0
\end{pmatrix}
\qquad\text{and}\qquad
W^\perp=\begin{pmatrix}
0      & \cdots & 0\\
\vdots &        & \vdots \\
0      & \cdots & 0\\
w_{k+1,1} & \cdots & w_{k+1,k}\\
\vdots &        & \vdots \\
w_{n,1} & \cdots & w_{n,k}\\
\end{pmatrix}.
\end{equation*}
The existence of this projection presupposes that the rank
of the matrix $W^\parallel$ on the left above is $k$, i.e.\/ the determinant
of the upper $k\times k$ block say $W_{\text{up}}$ is non-zero. 
This is not always true, we account for this momentarily. 
The subspace given by the span of the columns of $W^\parallel$ naturally 
coincides with $\Vb_0$. Indeed since $W^\parallel$ has rank $k$, 
there exists a rank $k$ transformation from $\Vb_0\to\Vb_0$, 
given by $W_{\text{up}}^{-1}\in\GL(\C^k)$, that transforms $W^\parallel$ to $W_0$.
Thus what distinguishes $\Wb$ from $\Vb_0$ is the form of $W^\perp$.
Under the same transformation of coordinates $W_{\text{up}}^{-1}$,
the lower $(n-k)\times k$ matrix say $W_{\text{low}}$ of $W^\perp$ 
becomes the $(n-k)\times k$ matrix $G\coloneqq W_{\text{low}}W_{\text{up}}^{-1}$. 
Or in other words if we perform this transformation of coordinates, the matrix $W$
as a whole becomes 
\begin{equation}\label{eq:canonicalcoordpatch}
\begin{pmatrix}
I_k\\
G
\end{pmatrix}.
\end{equation}
Thus any $k$-dimensional subspace $\Wb$ of $\C^n$ which can be projected onto $\Vb_0$
can be represented by this matrix. Conversely any $n\times k$ matrix of 
this form represents a $k$-dimensional subspace $\Wb$ of $\C^n$ that can be projected
onto $\Vb_0$. The matrix $G$ thus paramaterizes all the $k$-dimensional subspaces $\Wb$ 
that can be projected onto $\Vb_0$. As $G$ varies, the orientation of the 
subspace $\Wb$ within $\C^n$ varies.

What about the $k$-dimensional subspaces in $\C^n$ that cannot be projected
onto $\Vb_0$? This occurs when one or more of the column vectors of $W$ are
parallel to one or more of the orthogonal basis vectors 
$\{\mathrm{e}_{k+1},\ldots,\mathrm{e}_n\}$. Such subspaces cannot be 
represented in the form \eqref{eq:canonicalcoordpatch} above. Any such
matrices $W$ are rank $k$ matrices by choice, their
columns span a $k$-dimensional subspace $\Wb$ in $\C^n$, it's just
that they have a special orientation in the sense just described. 
We simply need to choose a better representation. Given a 
multi-index $\Sb=\{i_1,\ldots, i_k\}\subset \{1,\ldots,n\}$ of
cardinality $k$, let $\Vb_0(\Sb)$ denote the subspace given by 
$\text{span}\{\mathrm{e}_{i_1},\ldots,\mathrm{e}_{i_k}\}$.
The vectors $\{\mathrm{e}_{i}\}_{i\in\Sbc}$ span the subspace 
$\Vb_0^\perp(\Sb)$, the $(n-k)$-dimensional subspace of $\C^n$ 
orthogonal to $\Vb_0(\Sb)$. Since $W$ has rank $k$, there exists
a multi-index $\Sb$ such that the projection $\mathrm{pr}\colon\Wb\to\Vb_0(\Sb)$
exists. The arguments above apply with $\Vb_0(\Sb)$ replacing 
$\Vb_0=\Vb_0(\{1,\ldots,k\})$. The projections 
$\mathrm{pr}\colon\Wb\to\Vb_0(\Sb)$ and $\mathrm{pr}\colon\Wb\to\Vb_0^\perp(\Sb)$ 
respectively give 
\begin{equation*}
W^\parallel_{\Sb}=\begin{pmatrix}
W_{\Sb}\\
O_{\Sbc}
\end{pmatrix}
\qquad\text{and}\qquad
W^\perp_{\Sb}=\begin{pmatrix}
O_{\Sb}\\
W_{\Sbc}
\end{pmatrix}.
\end{equation*}
Here $W_{\Sb}$ represents the $k\times k$ matrix consisting of the
$\Sb$ rows of $W$ and so forth, and, for example, the form for $W^\parallel_{\Sb}$
shown is meant to represent the $n\times k$ matrix whose $\Sb$ rows
are occupied by $W_{\Sb}$ while the remaining rows contain zeros.
We can perform a rank $k$ transformation of coordinates 
$\Vb_0(\Sb)\to\Vb_0(\Sb)$ via $W_{\Sb}^{-1}\in\GL(\C^k)$
under which the matrix $W$ becomes 
\begin{equation}\label{eq:coordpatch}
\begin{pmatrix}
I_{\Sb}\\
G_{\Sbc}
\end{pmatrix}.
\end{equation}
Thus $G_{\Sbc}$ parameterizes all $k$-dimensional
subspaces $\Wb$ that can be projected onto $\Vb_0(\Sb)$.
Each possible choice of $\Sb$ generates a coordinate patch 
of the Grassmann manifold $\Gr(\C^n,\C^k)$.
For more details on establishing $\Gr(\C^n,\C^k)$ as a compact and connected
manifold, see Griffiths and Harris~\cite[p.~193-4]{GH}. 

Let us now consider the infinite dimensional extension to 
Fredholm Grassmann manifolds. They are also known as Sato Grassmannians,
Segal--Wilson Grassmannians, Hilbert--Schmidt Grassmannians and
restricted Grassmannians, as well as simply Hilbert Grassmannians.
See Sato~\cite{SatoI,SatoII}, Miwa, Jimbo and Date~\cite{MJD},
Segal and Wilson~\cite[Section~2]{SW}) and Pressley and Segal~\cite[Chapters~6,7]{PS}
for more details. In the infinite dimensional setting we suppose the 
underlying vector space is a separable Hilbert space $\Hb=\Hb(\C)$. 
Any separable Hilbert space is isomorphic to the sequence space $\ell^2=\ell^2(\C)$
of square summable complex sequences; see Reed and Simon~\cite[p.~47]{RS}. 
We will parameterize the $\C$-valued components of the sequences 
in $\ell^2=\ell^2(\C)$ by $\Nb$. 
This is sufficient as any sequence space $\ell^2=\ell^2(\Fb;\C)$, 
where $\Fb$ denotes a countable field isomorphic to $\Nb$ that
parameterizes the sequences therein, is isomorphic to $\ell^2=\ell^2(\Nb;\C)$. 
We recall any $\af\in\ell^2(\C)$ has the form 
$\af=\{\af(1), \af(2), \af(3),\ldots\}$ where $\af(n)\in\C$ for each $n\in\Nb$. 
Hereafter we represent such sequences by column vectors
$\af=(\af(1), \af(2), \af(3), \ldots)^{\mathrm{\tiny T}}$. Since we require 
square summability, we must have $\af^\dag \af=\sum_{n\in\Nb}\af^\ast(n)\,\af(n)<\infty$, 
where $\dag$ denotes complex conjugate transpose and $\ast$ denotes
complex conjugate only. We define the inner product 
$\la\,\cdot\,,\,\cdot\,\ra\colon\ell^2(\C)\otimes\ell^2(\C)\to\R$ by 
$\la \af, \mathfrak b\ra\coloneqq \sqrt{\af^\dag\mathfrak b}$ 
for any $\af, \mathfrak b\in\ell^2(\C)$. 
A natural complete orthonormal basis for $\ell^2(\C)$ 
is the \emph{canonical basis} $\{\ef_n\}_{n\in\Nb}$ where $\ef_n$
is the sequence whose $n$th component is one and all other components are zero.
We have the following corresponding definition for the Grassmannian
of all subspaces comparable in size to a given closed subspace 
$\Vb\subset\Hb$; see Segal and Wilson~\cite{SW} and Pressley and Segal~\cite{PS}.
\begin{definition}[Fredholm Grassmannian]\label{def:FredholmGrassmannian}
Let $\Hb$ be a separable Hilbert space with a given decomposition
$\Hb=\Vb\oplus\Vb^\perp$, where $\Vb$ and $\Vb^\perp$ are infinite
dimensional closed subspaces. The Grassmannian $\Gr(\Hb,\Vb)$
is the set of all subspaces $W$ of $\Hb$ such that:
\begin{enumerate}
\item The orthogonal projection $\mathrm{pr}\colon W\to\Vb$ is 
a Fredholm operator, indeed it is a Hilbert--Schmidt perturbation
of the identity; and
\item The orthogonal projection $\mathrm{pr}\colon W\to\Vb^\perp$ 
is a Hilbert--Schmidt operator.
\end{enumerate}
\end{definition}
Herein we exclusively assume that our underlying separable Hilbert
space $\Hb$ and closed subspace $\Vb$ are of the form
\begin{equation*}
\Hb\coloneqq\ell^2(\C)\times\ell^2_{\mathrm{d}}(\C)
\qquad\text{and}\qquad
\Vb\coloneqq\ell^2(\C),
\end{equation*}
where $\ell^2_{\mathrm{d}}(\C)$ is a closed subspace 
of $\ell^2(\C)$. We thus assume a special form for $\Hb$. 
This form is the setting for our applications discussed in our Introduction.
We use it to motivate the definition of the Fredholm Grassmannian 
above and its relation to our applications. Suppose we are given a set 
of independent sequences in $\ell^2(\C)\times\ell^2_{\mathrm{d}}(\C)$ 
which span $\ell^2(\C)$ and we 
record them as columns in the infinite matrix 
\begin{equation*}
W=\begin{pmatrix}
Q\\
P
\end{pmatrix}.
\end{equation*}
Here each column of $Q$ lies in $\ell^2(\C)$ 
and each column of $P$ lies in $\ell^2_{\mathrm{d}}(\C)$.
We denote by $\Wb$ the subspace of 
$\ell^2(\C)\times\ell^2_{\mathrm{d}}(\C)$ spanned by 
the columns of $W$. 
Let us denote by $\Vb_0$ the \emph{canonical subspace} 
which has the corresponding representation 
\begin{equation*}
W_0=\begin{pmatrix}
\id\\
O
\end{pmatrix},
\end{equation*}
where $\id=\id_{\ell^2(\C)}$. 
As above, suppose we are able to project $\Wb$ on $\Vb_0$. 
The projections $\mathrm{pr}\colon\Wb\to\Vb_0$ and 
$\mathrm{pr}\colon\Wb\to\Vb_0^\perp$ respectively give 
\begin{equation*}
W^\parallel=\begin{pmatrix}
Q\\
O
\end{pmatrix}
\qquad\text{and}\qquad
W^\perp=\begin{pmatrix}
O\\
P
\end{pmatrix}.
\end{equation*}
The existence of this projection presupposes that the determinant
of the upper block $Q$ is non-zero. We must now choose in which
sense we want this to hold. The columns of $Q$ are $\ell^2(\C)$-valued. 
We now retrospectively assume that we constructed $Q$ so that,
not only do its columns span $\ell^2(\C)$, it 
is also a Fredholm operator on $\ell^2(\C)$ of the 
form $Q=\id+Q^\prime$ where $Q^\prime\in\Jf_2\bigl(\ell^2(\C);\ell^2(\C)\bigr)$ 
and $\id=\id_{\ell^2(\C)}$. Here $\Jf_2\bigl(\ell^2(\C);\ell^2(\C)\bigr)$
is the class of Hilbert--Schmidt operators from $\ell^2(\C)\to\ell^2(\C)$,
equipped with the norm 
\begin{equation*}
\|Q^\prime\|_{\mathfrak J_2(\ell^2(\C);\ell^2(\C))}^2
\coloneqq\mathrm{tr}\,\bigl(Q^\prime\bigr)^\dag\bigl(Q^\prime\bigr),
\end{equation*}
where `$\mathrm{tr}$' represents the trace operator. 
For such Hilbert--Schmidt operators $Q^\prime$ we can define the
regularized Fredholm determinant 
\begin{equation*}
\mathrm{det}_2\bigl(\id+Q^\prime\bigr)
\coloneqq\exp\Biggl(\sum_{\ell\geqslant2}
\frac{(-1)^{\ell-1}}{\ell}\mathrm{tr}\,(Q^\prime)^\ell\Biggr).
\end{equation*}
The operator $Q=\id+Q^\prime$ is invertible if and only if 
$\mathrm{det}_2\bigl(\id+Q^\prime\bigr)\neq0$. For more
details see Simon~\cite{Simon:Traces}. Hence, assuming
that $Q^\prime\in\Jf_2\bigl(\ell^2(\C);\ell^2(\C)\bigr)$,
we can assert that the subspace given by the span of the 
columns of $W^\parallel$ coincides with the subspace spanned 
by $W_0$, i.e.\/ with $\Vb_0$. Indeed the transformation given by 
$Q^{-1}\in\GL\bigl(\ell^2(\C)\bigr)$ transforms $W^\parallel$
to $W_0$. Let us now focus on $W^\perp$. We now also
retrospectively assume that we constructed $P$ so that,
not only do its columns span $\ell^2_{\mathrm{d}}(\C)$, it is
a Hilbert--Schmidt operator from $\ell^2(\C)$ to
$\ell^2_{\mathrm{d}}(\C)$, i.e.\/ 
$P\in\Jf_2\bigl(\ell^2(\C);\ell^2_{\mathrm{d}}(\C)\bigr)$. 
Hence under the transformation of coordinates 
$Q^{-1}\in\GL\bigl(\ell^2(\C)\bigr)$ 
the matrix for $W$ becomes
\begin{equation*}
\begin{pmatrix}
\id\\
G
\end{pmatrix},
\end{equation*}
where $G\coloneqq PQ^{-1}$. Thus any subspace $\Wb$ that can be 
projected onto $\Vb_0$ can be represented in this way, and conversely.
The operator $G\in\Jf_2\bigl(\ell^2(\C);\ell^2_{\mathrm{d}}(\C)\bigr)$ 
thus parameterizes all subspaces $\Wb$ that can be projected onto $\Vb_0$.
We call the Fredholm index of the Fredholm operator $Q$ the 
\emph{virtual dimension} of $W$; see see Segal and Wilson~\cite{SW} 
and Pressley and Segal~\cite{PS} for more details.
\begin{remark}[Canonical coordinate patch]\label{rmk:canoncoordpatch}
In our applications we consider evolutionary flows in which
the operators $Q=Q(t)$ and $P=P(t)$ above evolve, as functions 
of time $t\geqslant0$, as solutions to linear differential equations. 
The initial data in all cases is taken to be $Q(0)=\id$ and 
$P(0)=G_0$ for some given data 
$G_0\in\Jf_2\bigl(\ell^2(\C);\ell^2_{\mathrm{d}}(\C)\bigr)$.
By assumption in general and by demonstration in practice, 
the flows are well-posed and smooth in time for $t\in[0,T]$ 
for some $T>0$. Hence there exists a time $T>0$
such that for $t\in[0,T]$ we know, by continuity, that $Q=Q(t)$
is an invertible Hilbert-Schmidt operator of virtual dimension zero and
of the form $Q(t)=\id+Q^\prime(t)$ where 
$Q^\prime(t)\in\Jf_2\bigl(\ell^2(\C);\ell^2(\C)\bigr)$.
For this time the flow for $Q=Q(t)$ and $P=P(t)$ prescibes a flow
for $G=G(t)$, with $G(t)=P(t)Q^{-1}(t)$. In addition, for this time,
whilst the orientation of the subspace prescribed by 
$Q=Q(t)$ and $P=P(t)$ evolves, the flow 
remains within the same coordinate patch of the Grassmannian 
$\Gr(\Hb,\Vb)$ prescribed by the initial data
as just described and explicitly outlined above.
\end{remark}
\begin{remark}[Frames]
More details on ``frames'' in the infinite dimensional context
can be found in Christensen~\cite{Christensen} and Balazs~\cite{Balazs}. 
\end{remark}
There are three possible obstructions to the construction of the 
class of subspaces above as follows, the: (i) Virtual dimension of 
$\Wb$, i.e.\/ the Fredholm index of $Q$, may differ by an integer 
value; (ii) Operator $Q^\prime$ may not be Hilbert--Schmidt 
valued---it could belong to a `higher' Schatten--von Neumann class; or 
(iii) Determinant of $Q$ may be zero. The consequences of these 
issues for connected components, submanifolds and coordinate patches of
$\Gr(\Hb,\Vb)$ are covered in detail in general in Pressley
and Segal~\cite[Chap.~7]{PS}. These have important implications
for regularity of the flows mentioned in Remark~\ref{rmk:canoncoordpatch} above,
i.e.\/ for our applications. However we leave these questions
for further investigation, see Section~\ref{sec:conclu}. 
Suffice to say for the moment, from Pressley and Segal~\cite[Prop.~7.1.6]{PS},
we know that given any subspace $\Wb$ of $\Hb$ there exists a
representation analogous to the general coordinate patch 
form~\eqref{eq:coordpatch} with $\Sb$ a suitable countable set.
In other words there exists a subspace cover. More details
on infinite dimensional Grassmannians can be found in
Sato~\cite{SatoI,SatoII}, Abbondandolo and Majer~\cite{AM}
and Furitani~\cite{F}.

We have introduced the Fredholm Grassmannian here in the context
where the underlying Hilbert space is 
$\Hb=\ell^2(\C)\times\ell^2_{\mathrm{d}}(\C)$ and 
the subspace $\Vb\cong\ell^2(\C)$. In our applications
the context will be $\Hb=L^2(\Ib;\C)\times L_{\mathrm{d}}^2(\Ib;\C)$
and $\Vb\cong L^2(\Ib;\C)$ where the continuous interval 
$\Ib\subseteq\R$. We include here the cases when $\Ib$ is finite,
semi-infinite of the form $[a,\infty)$ for some real constant $a$
or the whole real line. As above, here $L_{\mathrm{d}}^2(\Ib;\C)$
denotes a closed subspace of $L^2(\Ib;\C)$---corresponding to
intersection of the domains of the unbounded operators $D$ and $B$
in our applications. All such spaces $L^2(\Ib;\C)$ are separable 
and isomorphic to $\ell^2(\C)$, and correspondingly for the closed
subspaces. See for example Christensen~\cite{Christensen} 
or Blanchard and Br\"uning~\cite{BB-H} for more details. 
It is straightforward to transfer statements we have made 
thusfar for the Fredholm Grassmannian
in the square-summable sequence space context across to the 
square integrable function space context. When 
$\Hb=L^2(\Ib;\C)\times L_{\mathrm{d}}^2(\Ib;\C)$
and $\Vb\cong L^2(\Ib;\C)$ the operators $Q^\prime$ 
and $P$ are Hilbert--Schmidt operators in the sense
that $Q^\prime\in\mathfrak J_2\bigl(L^2(\Ib;\C);L^2(\Ib;\C)\bigr)$ and 
$P\in\mathfrak J_2\bigl(L^2(\Ib;\C);L_{\mathrm{d}}^2(\Ib;\C)\bigr)$. 
By standard theory such Hilbert--Schmidt operators can be parameterized
via integral kernel functions, say, $q^\prime\in L^2(\Ib^2;\C)$ and 
$p\in L^2(\Ib;L^2_{\mathrm{d}}(\Ib;\C))$ and their actions represented by
\begin{align*}
Q^\prime(f)(x)&=\int_{\Ib}q^\prime(x,y)\,f(y)\,\rd y,\\
P(f)(x)&=\int_{\Ib}p(x,y)\,f(y)\,\rd y,
\end{align*}
for any $f\in L^2(\Ib;\C)$ and where $x\in\Ib$. Furthermore we know 
we have the isometries $\|Q^\prime\|_{\mathfrak J_2(L^2(\Ib;\C);L^2(\Ib;\C))}
=\|q^\prime\|_{L^2(\Ib^2;\C)}$
and $\|P\|_{\mathfrak J_2(L^2(\Ib;\C);L_{\mathrm{d}}^2(\Ib;\C))}
=\|p\|_{L^2(\Ib;L^2_{\mathrm{d}}(\Ib;\C))}$; 
see Reed and Simon~\cite[p.~210]{RS}. Hence the subspace
$\Wb$ above and its representation in the canonical coordinate
patch are given by
\begin{equation*}
W=\begin{pmatrix}
q\\
p
\end{pmatrix}
\quad 
\rightsquigarrow
\quad
\begin{pmatrix}
\delta\\
g
\end{pmatrix}.
\end{equation*}
Here we suppose $q(x,y)=\delta(x-y)+q'(x,y)$ with $\delta(x-y)$
representing the identity operator in $L^2(\Ib;\C)$ at the integral kernel level.
The function $g=g(x,y)$ is the $L^2(\Ib;L^2_{\mathrm{d}}(\Ib;\C))$-valued
kernel associated with the Hilbert--Schmidt operator $G$. It is 
explicitly obtained by solving the Fredholm equation given by
\begin{equation*}
p(x,y)=\int_\Ib g(x,z)\,q(z,y)\,\rd z.
\end{equation*}
Solving this equation for $g$ is equivalent to solving
the operator relation $P=G\,Q$ for $G$ by postcomposing by $Q^{-1}$.

\section{Fredholm Grassmannian flows}\label{sec:flows} 
We show how linear evolutionary flows on subspaces of an abstract separable 
Hilbert space $\Hb$ generate a quadratically nonlinear flow on a 
coordinate patch of an associated Fredholm Grassmann manifold. The
setting is similar to that outlined at the beginning of the last section. 
Assume for the moment that $\Hb$ admits a direct sum orthogonal decomposition 
$\Hb=\Vb\oplus\Vb^\perp$, where $\Vb$ and $\Vb^\perp$ are closed subspaces of $\Hb$. 
The subspace $\Vb$ is fixed. Now suppose there exists a time $T>0$ such that
for each time $t\in[0,T]$ there exists a continuous path of subspaces $\Wb=\Wb(t)$ 
of $\Hb$ such that the projections $\mathrm{pr}\colon\Wb(t)\to\Vb$ and 
$\mathrm{pr}\colon\Wb(t)\to\Vb^\perp$ can be respectively parameterised by
the operators $Q(t)=\id+Q^\prime(t)$ and $P(t)$. We in fact assume the path 
of subspaces $\Wb(t)$ is smooth in time and $Q^\prime(t)$ and $P(t)$ are 
Hilbert--Schmidt operators so that indeed 
$Q^\prime\in C^\infty\bigl([0,T];\Jf_2(\Vb;\Vb)\bigr)$ and
$P\in C^\infty\bigl([0,T];\mathrm{Dom}(D)\cap\mathrm{Dom}(B)\bigr)$.
Here $D$ and $B$ are in general unbounded operators for which, as 
we see presently, for each $t\in[0,T]$ we require $DP(t)\in\Jf_2(\Vb;\Vb^\perp)$ 
and $BP(t)\in\Jf_2(\Vb;\Vb)$. The subspaces  
$\mathrm{Dom}(D)\subseteq\Jf_2(\Vb;\Vb^\perp)$
and $\mathrm{Dom}(B)\subseteq\Jf_2(\Vb;\Vb)$ are their 
respective domains. Our analysis also involves two bounded
operators $A\in C^\infty\bigl([0,T];\Jf_2(\Vb;\Vb)\bigr)$
and $C\in C^\infty\bigl([0,T];\Jf_2(\Vb;\Vb^\perp)\bigr)$.
The evolution of $Q=Q(t)$ and $P=P(t)$ is prescribed by
the following system of differential equations.
\begin{definition}[Linear Base and Auxiliary Equations]
We assume there exists a $T>0$ such that, for the linear operators 
$A$, $B$, $C$ and $D$ described above, the linear operators
$Q^\prime\in C^\infty\bigl([0,T];\Jf_2(\Vb;\Vb)\bigr)$ and
$P\in C^\infty\bigl([0,T];\mathrm{Dom}(D)\cap\mathrm{Dom}(B)\bigr)$
satisfy the linear system of operator equations
\begin{align*}
\pa_tQ&=AQ+BP,\\
\pa_tP&=CQ+DP,
\end{align*}
where $Q=\id+Q^\prime$. We assume at time $t=0$ that $Q^\prime(0)=O$ 
so that $Q(0)=\id$ and $P(0)=P_0$ for some given 
$P_0\in\mathrm{Dom}(D)\cap\mathrm{Dom}(B)$. 
We call the evolution equation for $P=P(t)$ the \emph{base equation} 
and that for $Q=Q(t)$ the \emph{auxiliary equation}.
\end{definition}
\begin{remark}
The initial condition $Q(0)=\id$ and $P(0)=P_0$ means that the
corresponding subspace $\Wb(0)$ is represented in the canonical
coordinate chart of $\Gr(\Hb,\Vb)$. Hereafter we will assume
that for $t\in[0,T]$ the subspace $\Wb(t)$ is representable 
in the canonical coordinate chart and in particular that 
$\mathrm{det}_2Q(t)\neq0$.
\end{remark}
The base and auxiliary equations represent two essential 
ingredients in our prescription, which to be complete, 
requires a third crucial ingredient. This is to propose 
a relation between $P$ and $Q$ as follows.
\begin{definition}[Riccati Relation]
We assume there exists a $T>0$ such that for 
$P\in C^\infty\bigl([0,T];\mathrm{Dom}(D)\cap\mathrm{Dom}(B)\bigr)$
and $Q^\prime\in C^\infty\bigl([0,T];\Jf_2(\Vb;\Vb)\bigr)$
there exists a linear operator 
$G\in C^\infty\bigl([0,T];\mathrm{Dom}(D)\cap\mathrm{Dom}(B)\bigr)$
satisfying the linear Fredholm equation
\begin{equation*}
P=G\,Q,
\end{equation*}
where $Q=\id+Q^\prime$. We call this the \emph{Riccati Relation}.
\end{definition}
Given solution linear operators $P=P(t)$ and $Q=Q(t)$ to the linear base 
and auxiliary equations we can prove the existence of a suitable
solution $G=G(t)$ to the linear Fredholm equation constituting the 
Riccati relation. This result is proved in 
Beck \textit{et al.}~\cite{BDMStrans}. The result is as follows.
\begin{lemma}[Existence and Uniqueness: Riccati relation]\label{lemma:EandU}
Assume there exists a $T>0$ such that
$P\in C^{\infty}\bigl([0,T];\mathrm{Dom}(D)\cap\mathrm{Dom}(B)\bigr)$,
$Q^\prime\in C^{\infty}\bigl([0,T];\mathfrak J_2(\Vb;\Vb)\bigr)$
and $Q^\prime(0)=O$. Then there exists a $T^\prime>0$ with $T^\prime\leqslant T$ 
such that for $t\in[0,T^\prime]$ we have 
$\mathrm{det}_2\bigl(Q(t)\bigr)\neq0$ and
$\|Q^\prime(t)\|_{\mathfrak J_2(\Vb;\Vb)}<1$.
In particular, there exists a unique solution 
$G\in C^{\infty}\bigl([0,T^\prime];\mathrm{Dom}(D)\cap\mathrm{Dom}(B)\bigr)$
to the Riccati relation.
\end{lemma}
The proof utilizes the fact that we assume the solutions
$P(t)\in\mathrm{Dom}(D)\cap\mathrm{Dom}(B)$ and $Q(t)\in\mathfrak J_2(\Vb;\Vb)$ 
are smooth in time and at time $t=0$ we have $\mathrm{det}_2\bigl(Q(0)\bigr)=1$
and $\|Q^\prime(0)\|_{\mathfrak J_2(\Vb;\Vb)}=0$. Hence for a short time 
we are guaranteed that $\mathrm{det}_2\bigl(Q(t)\bigr)$ is non-zero and 
$\|Q^\prime(t)\|_{\mathfrak J_2(\Vb;\Vb)}$ is sufficiently small to provide
suitable bounds on $G(t)=P(t)Q^{-1}(t)$. Our main result now is 
that the solution $G=G(t)$ to the Riccati relation satisfies a 
quadratically nonlinear evolution equation as follows.
\begin{theorem}[Riccati evolution equation]\label{thm:main}
Suppose we are given the initial data $G_0\in\mathrm{Dom}(D)\cap\mathrm{Dom}(B)$ 
and that $Q^\prime(0)=O$ and $P(0)=G_0$. Assume for some $T>0$ that 
$P\in C^{\infty}\bigl([0,T];\mathrm{Dom}(D)\cap\mathrm{Dom}(B)\bigr)$,
$Q^\prime\in C^{\infty}\bigl([0,T];\mathfrak J_2(\Vb;\Vb)\bigr)$
satisfy the linear base and auxiliary equations and that 
$G\in C^{\infty}\bigl([0,T];\mathrm{Dom}(D)\cap\mathrm{Dom}(B)\bigr)$
solves the Riccati relation. Then this solution $G$ to the Riccati
relation necessarily satisfies $G(0)=G_0$ and for $t\in[0,T]$ solves 
the Riccati evolution equation 
\begin{equation*}
\pa_tG=C+DG-G\,(A+BG).
\end{equation*}
\end{theorem}
\begin{proof}
By direct computation, if we differentiate the Riccati
relation $P=G\,Q$ with respect to time using the product rule and 
use that $P$ and $Q$ satisfy the linear base and auxiliary equations
we find 
$(\pa_tG)\,Q=\pa_t P-G\,\pa_tQ=DP-G\,(AQ+BP)=(DG)\,Q-G\,(A+BG)\,Q$.
Postcomposing by $Q^{-1}$ establishes the result.\qed
\end{proof}
We now consider the abstract development above at the partial
differential equation level with an eye towards our applications. 
All our assumptions hitherto in this section apply here as well.
As hinted in our Introduction and indicated more explicitly 
at the end of Section~\ref{sec:Grassmannians}, suppose our
underlying separable Hilbert space is 
$\Hb=L^2(\Ib;\R)\times L_{\mathrm{d}}^2(\Ib;\R)$,
where the continuous interval $\Ib\subseteq\R$. 
The function space $L_{\mathrm{d}}^2(\Ib;\R)\subseteq L^2(\Ib;\R)$
is a subspace of $L^2(\Ib;\R)$ which we will explicitly 
define presently. We assume the closed subspace $\Vb$ of $\Hb$
to be $\Vb\cong L^2(\Ib;\R)$. By assumption we know
for some $T>0$ the operators $Q^\prime$ 
and $P$ are Hilbert--Schmidt operators with
$Q^\prime\in C^\infty\bigl([0,T];\mathfrak J_2\bigl(L^2(\Ib;\R);L^2(\Ib;\R)\bigr)\bigr)$
and
$P\in C^\infty\bigl([0,T];\mathfrak J_2\bigl(L^2(\Ib;\R);L_{\mathrm{d}}^2(\Ib;\R)\bigr)\bigr)$.
By standard theory the actions $Q^\prime$ and $P$ can
be represented by the integral kernel functions $q^\prime=q^\prime(x,y;t)$
and $p=p(x,y;t)$, respectively where
\begin{equation*}
q^\prime\in C^\infty\bigl([0,T];L^2(\Ib^2;\R)\bigr)
\qquad\text{and}\qquad
p\in C^\infty\bigl([0,T];L^2\bigl(\Ib;L^2_{\mathrm{d}}(\Ib;\R)\bigr)\bigr).
\end{equation*}
Henceforth we assume that the operator $B$ is multiplicative
corresponding to multiplcation by the smooth, bounded, square-integrable and
real-valued function $b=b(x)$. We also assume 
the operator $D$ is the unbounded operator $d=d(\pa_x)$ which is a polynomial   
of $\pa_x$ with constant real-valued coefficients. We can now specify
$L_{\mathrm{d}}^2(\Ib;\R)\subseteq L^2(\Ib;\R)$, it corresponds to
the domain of the operator $d=d(\pa_x)$. 
Also by assumption $A$ and $C$ are Hilbert--Schmidt valued operators and can thus be
represented by integral kernel functions $a=a(x,y;t)$ and $c=c(x,y;t)$,
respectively, where $a\in C^\infty\bigl([0,T];L^2(\Ib^2;\R)\bigr)$
and $c\in C^\infty\bigl([0,T];L^2\bigl(\Ib;L^2_{\mathrm{d}}(\Ib;\R)\bigr)\bigr)$
The linear base and auxiliary equations, since $Q(t)=\id+Q^\prime(t)$, 
thus have the form
\begin{subequations}\label{eq:linearbaseandauxpdes}
\begin{align}
\pa_t q^\prime(x,y;t)&=a(x,y;t)+\int_\Ib a(x,z;t)q^\prime(z,y;t)\,\rd z+b(x)\,p(x,y;t),\\
\pa_t p(x,y;t)&=c(x,y;t)+\int_\Ib c(x,z;t)q^\prime(z,y;t)\,\rd z+d(\pa_x)\,p(x,y;t).
\end{align}
\end{subequations}
\begin{remark}\label{rmk:diffdisp}
To be consistent with our assumptions on the properties of $P=P(t)$
and its corresponding integral kernel $p=p(x,y;t)$ for $t\in[0,T]$
as outlined above, we must suitably restrict the choice of the 
class of operator $d=d(\pa_x)$ appearing in the base equation.
In addition to the class properties outlined just above, we 
assume henceforth, and in particular for all our applications in 
Section~\ref{sec:examples}, that $d=d(\pa_x)$ is diffusive or dispersive
as a polynomial operator in $\pa_x$. Hence for example we could
assume that $d$ is a polynomial of only even degree terms in $\pa_x$
with the $2N$th degree term having a real coefficient of sign $(-1)^{N+1}$.
Alternatively for example $d$ could have a dispersive form such as $d=-\pa_x^3$.
\end{remark}
We are now in a position to prove our main result for
evolutionary partial differential equations with nonlocal
quadratic nonlinearities.
\begin{corollary}[Grassmannian evolution equation]\label{cor:main}
Given the initial data 
$g_0\in C^\infty\bigl(\Ib^2;\R\bigr)\cap L^2\bigl(\Ib;L^2_{\mathrm{d}}(\Ib;\R)\bigr)$
suppose $q^\prime=q^\prime(x,y;t)$ and $p=p(x,y;t)$ are the solutions
to the linear evolutionary base and auxiliary equations~\eqref{eq:linearbaseandauxpdes}
with $p(x,y;0)=g_0(x,y)$ and $q^\prime(x,y;0)=0$. 
Suppose the operator $d=d(\pa_x)$ is of the diffusive or 
dispersive form described in Remark~\ref{rmk:diffdisp}.
Then there is a $T>0$ such that the solution 
$g\in C^\infty\bigl([0,T];L^2\bigl(\Ib;L^2_{\mathrm{d}}(\Ib;\R)\bigr)\bigr)$
to the linear Fredholm equation
\begin{equation}\label{eq:Fredholmeq}
p(x,y;t)=g(x,y;t)+\int_{\Ib}g(x,z;t)\,q^\prime(z,y;t)\,\rd z,
\end{equation}
solves the evolutionary partial differential equation
with quadratic nonlocal nonlinearities of the form
\begin{equation*}
\pa_tg(x,y;t)=c(x,y;t)+d(\pa_x)\,g(x,y;t)
-\!\int_{\Ib}g(x,z;t)\bigl(a(z,y;t)+b(z)\,g(z,y;t)\bigr)\,\rd z.
\end{equation*}
\end{corollary}
\begin{proof}
That for some $T>0$ there exists a solution 
$g\in C^\infty\bigl([0,T];L^2\bigl(\Ib;L^2_{\mathrm{d}}(\Ib;\R)\bigr)\bigr)$
to the linear Fredholm equation~\eqref{eq:Fredholmeq}, 
i.e.\/ the Riccati relation, follows 
from Lemma~\ref{lemma:EandU} and our assumptions on $q^\prime=q^\prime(x,y;t)$
and $p=p(x,y;t)$ outlined above. That this solution $g$ also
solves the evolutionary partial differential equation
with quadratic nonlocal nonlinearity shown 
follows from Theorem~\ref{thm:main}. \qed
\end{proof}
It is instructive to see the proof of the second of the
results from Corollary~\ref{cor:main} at the integral kernel
level, i.e.\/ the proof that the solution $g$ to the 
linear Fredholm equation~\eqref{eq:Fredholmeq} also solves 
the evolutionary partial differential equation
with quadratic nonlocal nonlinearity shown. 
We present this here. First we differentiate the 
linear Fredholm equation~\eqref{eq:Fredholmeq}
with respect to time. This generates the relation 
\begin{equation*}                                                      
\pa_tg(x,y;t)+\int_{\Ib}\pa_tg(x,z;t)\, q^\prime(z,y;t)\,\rd z              
=\pa_tp(x,y;t)-\int_{\Ib}g(x,z;t)\,\pa_tq^\prime(z,y;t)\,\rd z.          
\end{equation*}                                         

Second we substitute for $\pa_tq^\prime$ and $\pa_tp$ using their evolution equations. 
Let us consider the first term on the right above. We find that           
\begin{align*}                
\pa_tp(x,y;t)
=&\;\int_{\Ib}c(x,z;t)\,\bigl(\delta(z-y)+q^\prime(z,y;t)\bigr)\,\rd z+d(\pa_x)p(x,y;t)\\     
=&\;\int_{\Ib}c(x,z;t)\,\bigl(\delta(z-y)+q^\prime(z,y;t)\bigr)\,\rd z\\
&\;+d(\pa_x)\biggl(g(x,y;t)+\int_{\Ib}g(x,z;t)\,q^\prime(z,y;t)\,\rd z \biggr)\\
=&\;\int_{\Ib}c(x,z;t)\,\bigl(\delta(z-y)+q^\prime(z,y;t)\bigr)\,\rd z\\
&\;+d(\pa_x)\int_{\Ib}g(x,z;t)\,\bigl(\delta(z-y)+q^\prime(z,y;t)\bigr)\,\rd z\\
=&\;\int_{\Ib}\bigl(c(x,z;t)
+d(\pa_x)g(x,z;t)\bigr)\bigl(\delta(z-y)+q^\prime(z,y;t)\bigr)\,\rd z.
\end{align*}                        
Now consider the second term on the right above.
We observe
\begin{align*}
\int_{\Ib}&g(x,z;t)\,\pa_tq^\prime(z,y;t)\,\rd z\\
=&\;\int_{\Ib}g(x,z;t)\,\biggl(
\int_{\Ib}a(z,\zeta;t)\,\bigl(\delta(\zeta-y)+q^\prime(\zeta,y;t)\bigr)\,\rd\zeta\biggr)\,\rd z\\
&\;+\int_{\Ib}g(x,z;t)\,\bigl(b(z)p(z,y;t)\bigr)\,\rd z\\
=&\;\int_{\Ib}g(x,z;t)\,\biggl(
\int_{\Ib}a(z,\zeta;t)\,\bigl(\delta(\zeta-y)+q^\prime(\zeta,y;t)\bigr)\,\rd\zeta\biggr)\,\rd z\\
&\;+\int_{\Ib}g(x,z;t)\,\biggl(b(z)
\int_{\Ib}g(z,\zeta;t)\,\bigl(\delta(\zeta-y)+q^\prime(\zeta,y;t)\bigr)\,\rd\zeta\biggr)\,\rd z\\
=&\;\int_{\Ib}\biggl(\int_{\Ib}g(x,\zeta;t)\bigl(a(\zeta,z;t)
+b(\zeta)g(\zeta,z;t)\bigr)\,\,\rd\zeta\biggr)
\bigl(\delta(z-y)+q^\prime(z,y;t)\bigr)\,\rd z.
\end{align*}
Putting these results together and post-composing by the operator 
$Q^{-1}$ generates the required result. Another way to enact this last
step is to postmultiply the final combined result by 
`$\delta(y-\eta)+\tilde q^\prime(y,\eta;t)$' for some $\eta\in\Ib$. 
This is the kernel associated with the inverse operator $Q^{-1}=\id+\tilde Q^\prime$
to $Q=\id+Q^\prime$. Then integrating over $y\in\Ib$ gives the result 
for $g=g(x,\eta;t)$.
\begin{remark}[Nonlocal nonlinearities with derivatives]
In the linear base and auxiliary equations~\eqref{eq:linearbaseandauxpdes}
we could take $b$ to be a constant coefficient polynomial of $\pa_x$.
With minor modifications the results we derive above still apply.
\end{remark}
We now need to demonstrate as a practical procedure, how 
linear evolutionary partial differential equations for $p$ and $q^\prime$ 
generate solutions to the evolutionary partial differential equation 
with quadratic nonlocal nonlinearities at hand. We show this explicitly
through two examples in the next section.

\section{Examples}\label{sec:examples}
We now consider some evolutionary partial differential equations
with nonlocal quadratic nonlinearities and explicitly show how
to generate solutions to them from the linear base and auxiliary equations
and linear Riccati relation. In both examples we take $\Ib\coloneqq\R$.
Note that throughout we define the Fourier transform for any given function $f=f(x)$ 
and its inverse as 
\begin{equation*}
\hat{f}(k)\coloneqq\int_\R f(x)\mathrm{e}^{2\pi\mathrm{i}kx}\,\rd x
\qquad\text{and}\qquad
f(x)\coloneqq\int_\R\hat{f}(k)\mathrm{e}^{-2\pi\mathrm{i}kx}\,\rd k.
\end{equation*}
\begin{example}[Nonlocal convolution nonlinearity]\label{ex:nonlocal}
In this case the target evolutionary partial differential equation
has a quadratic nonlinearity in the form of a convolution and is
given by
\begin{equation*}
\pa_tg=d\,g-g\star g,
\end{equation*}
where $d=d(\pa_x)$ and the $\star$ operation here does 
indeed represent convolution. In other words for this example we suppose
\begin{equation*}
\bigl(g\star g\bigr)(x;t)=\int_\R g(x-z;t)\,g(z;t)\,\rd z.
\end{equation*}
We assume smooth and square-integrable initial data $g_0=g_0(x)$.

To find solutions via our approach, we begin by assuming the 
kernel $g$ of the operator $G$ has the convolution form $g=g(x-y;t)$. 
We further assume the linear base and auxiliary equations have the form 
\begin{align*}
\pa_t p(x,y;t)&=d(\pa_x)\,p(x,y;t),\\
\pa_t q^\prime(x,y;t)&=b(x)\,p(x,y;t),
\end{align*}
with in fact $b\equiv 1$. In addition we suppose
$d=d(\pa_x)$ is of diffusive or dispersive form as described
in Remark~\ref{rmk:diffdisp}. 
In this case the Grassmannian evolution equation 
in Corollary~\ref{thm:main} has the form
\begin{equation*}
\pa_tg(x-y;t)=d(\pa_x)\,g(x-y;t)-\int_{\R}g(x-z;t)\,g(z-y;t)\,\rd z,
\end{equation*}
which by setting $y=0$ matches the system under consideration. 
We verify the sufficient conditions for Corollary~\ref{thm:main}
to apply presently. In Fourier space our example 
partial differential equation naturally takes the form
\begin{equation}\label{eq:nonlocal}
\pa_t\hat g=d(2\pi\mathrm{i}k)\,\hat g-\hat g^2.
\end{equation}

We generate solutions to the given partial differential equation for $g$ 
from the linear base and auxiliary equations, for the given initial data $g_0$, 
as follows. Note the base equation has the following equivalent form and solution 
in Fourier space:
\begin{equation*}
\pa_t\hat p(k,y;t)=d(2\pi\mathrm{i}k)\,\hat p(k,y;t) \quad\Leftrightarrow\quad
\hat p(k,y;t)=\mathrm{e}^{d(2\pi\mathrm{i}k)\,t}\,\hat p_0(k,y).
\end{equation*}
Here $\hat p_0$ is the Fourier transform of the initial data for $p$.
In Fourier space the auxiliary equation has the form and solution:
\begin{equation*}
\pa_t\hat q^\prime(k,y;t)=\hat p(k,y;t) \quad\Leftrightarrow\quad
\hat q^\prime(k,y;t)-\hat q_0^\prime(k,y)
=\frac{\mathrm{e}^{d(2\pi\mathrm{i}k)\,t}-1}{d(2\pi\mathrm{i}k)}\,\hat p_0(k,y).
\end{equation*}
Here $\hat q_0^\prime(k,y)$ is the Fourier transform of the initial data 
for $\hat q^\prime$. As per the general theory, 
we suppose $\hat q_0^\prime(k,y)=0$. 
This means if we set $t=0$ in the Riccati
relation we find 
\begin{equation*}
p_0(x,y)=g_0(x-y)\qquad\Leftrightarrow\qquad
\hat p_0(k,y)=\mathrm{e}^{2\pi\mathrm{i}ky}\,\hat g_0(k).
\end{equation*}
where $g_0$ is the initial data for the partial differential equation for $g$.
Hence explicitly we have 
\begin{equation*}
\hat p(k,y;t)=\mathrm{e}^{d(2\pi\mathrm{i}k)\,t}\,\mathrm{e}^{2\pi\mathrm{i}ky}\,\hat g_0(k)
\quad\text{and}\quad
\hat q^\prime(k,y;t)=\frac{\mathrm{e}^{d(2\pi\mathrm{i}k)\,t}-1}
{d(2\pi\mathrm{i}k)}\,\mathrm{e}^{2\pi\mathrm{i}ky}\,\hat g_0(k).
\end{equation*}
Note by taking the inverse Fourier transform, we deduce that 
$p=p(x-y;t)$ and $q^\prime=q^\prime(x-y;t)$. From these explicit forms
for their Fourier transforms, we deduce there exists
a $T>0$ such that on the time interval $[0,T]$ we know
$p$ and $q^\prime$ have the regularity required so that
Corollary~\ref{thm:main} applies.
Further, the Riccati relation in this case is 
\begin{align*}
&&p(x,y;t)&=g(x-y;t)+\int_\R g(x-z;t)\,q^\prime(z,y;t)\,\rd z\\
\qquad\Leftrightarrow\qquad&&
\hat p(k,y;t)&=\hat g(k;t)\bigl(\mathrm{e}^{2\pi\mathrm{i}ky}+\hat q^\prime(k,y;t)\bigr).
\end{align*}
Thus using the expressions for $\hat p$ and $\hat q^\prime$ above we find that
\begin{equation*}
\hat g(k;t)
=\frac{\mathrm{e}^{d(2\pi\mathrm{i}k)\,t}\,\hat g_0(k)}
{1+\Bigl(\bigl(\mathrm{e}^{d(2\pi\mathrm{i}k)\,t}-1\bigr)/d(2\pi\mathrm{i}k)\Bigr)\,\hat g_0(k)}.
\end{equation*}
Direct substitution into the Fourier form~\eqref{eq:nonlocal} 
of our example partial differential equation verifies it is 
indeed the solution for the initial data $g_0$.

In Figure~\ref{fig:nonlocal} we show the solution to the nonlocal
quadratically nonlinear partial differential equation above, for 
$d=\partial_x^2+1$ and a given generic initial profile $g_0$. 
The left panel shows the evolution of the 
solution profile computed using a direct integration approach. 
By this we mean we approximated $\partial_x^2$ by the central
difference formula and computed the nonlinear convolution by
computing the inverse Fourier transform of $\bigl(\hat g(k)\bigr)^2$.
We used the inbuilt Matlab integrator \texttt{ode23s} to integrate
in time. Similar direct integration could be achieved by integrating
the differential equation~\eqref{eq:nonlocal} for $\hat g$ using \texttt{ode23s} and then
computing the inverse Fourier transform. The right panel in  
Figure~\ref{fig:nonlocal} shows the solution evolution computed using
our Riccati approach. As expected, the solutions look
identical (up to numerical precision), even when we continue the 
solution past the time when the diffusion has reached the boundaries 
of the finite domain of integration in $x$, roughly
half way along the interval of evolution shown.  
\end{example}

\begin{figure}
  \begin{center}
  \includegraphics[width=5.5cm,height=5.5cm]{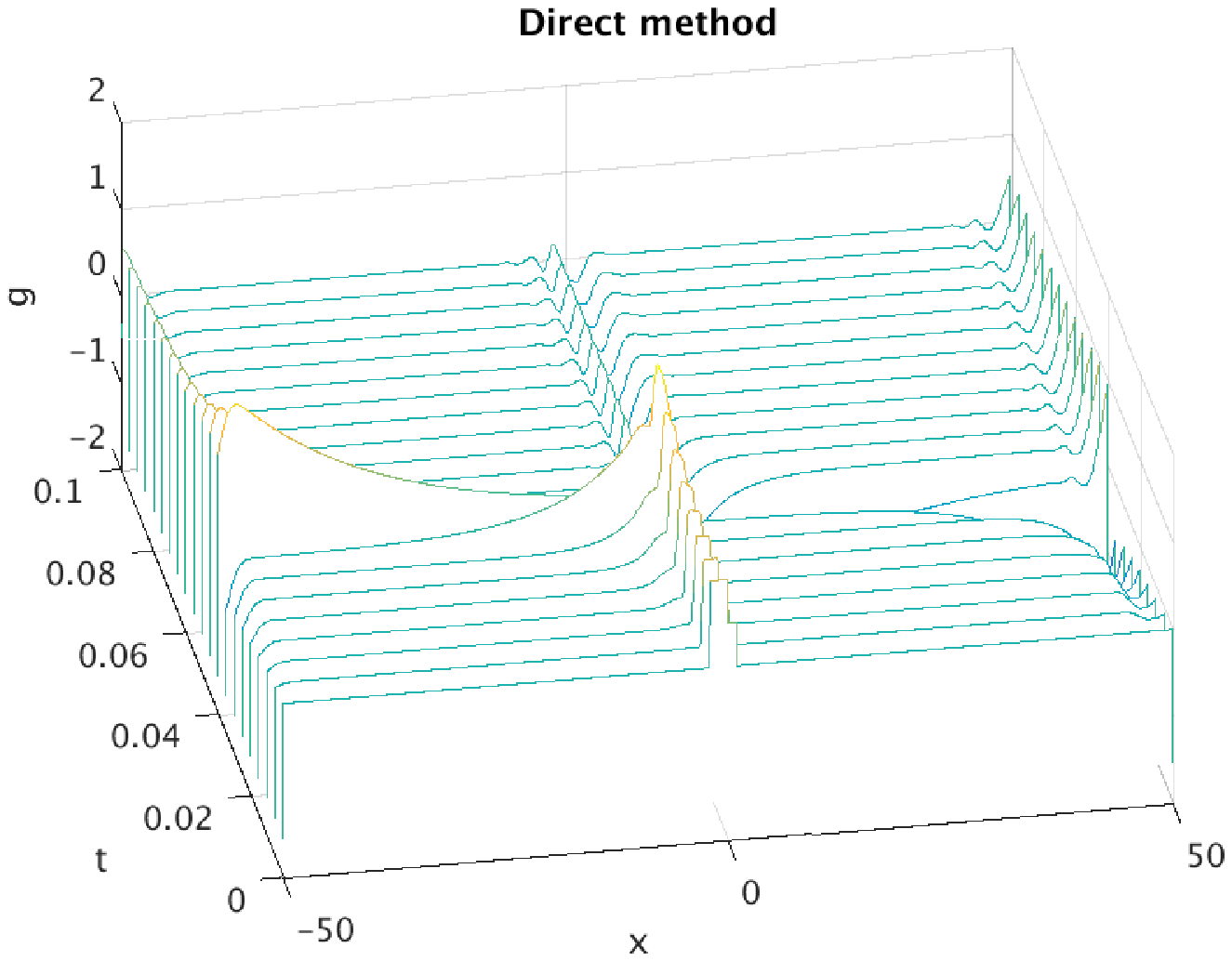}
  \includegraphics[width=5.5cm,height=5.5cm]{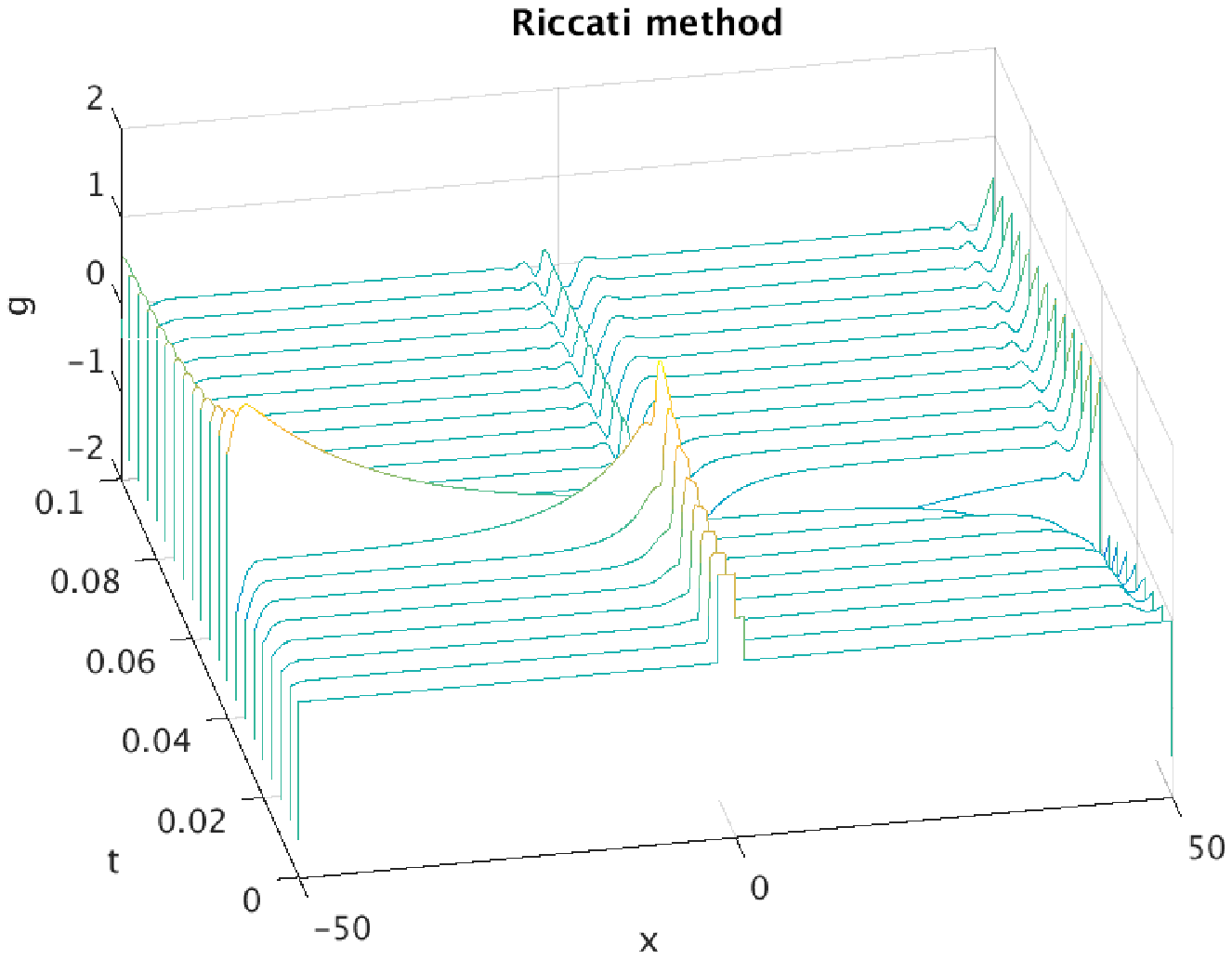}
  \end{center}
  \caption{We plot the solution to the nonlocal quadratically 
nonlinear partial differential equation from Example~\ref{ex:nonlocal}. 
We used a generic initial profile $g_0$ as shown. The left panel shows
the solution computed using a direct integration approach while the right
panel shows the solution computed using our Riccati approach.}
\label{fig:nonlocal}
\end{figure}

\begin{remark}[Multi-dimensions]\label{rmk:multi-d}  
This last example extends to the case where $x,y\in\R^n$ for any $n\geqslant1$ 
when $d$ is a scalar operator such as a power of the Laplacian, 
with $p$, $q^\prime$ and $g$ all scalar. 
\end{remark}

\begin{figure}
  \begin{center}
  \includegraphics[width=6cm,height=6cm]{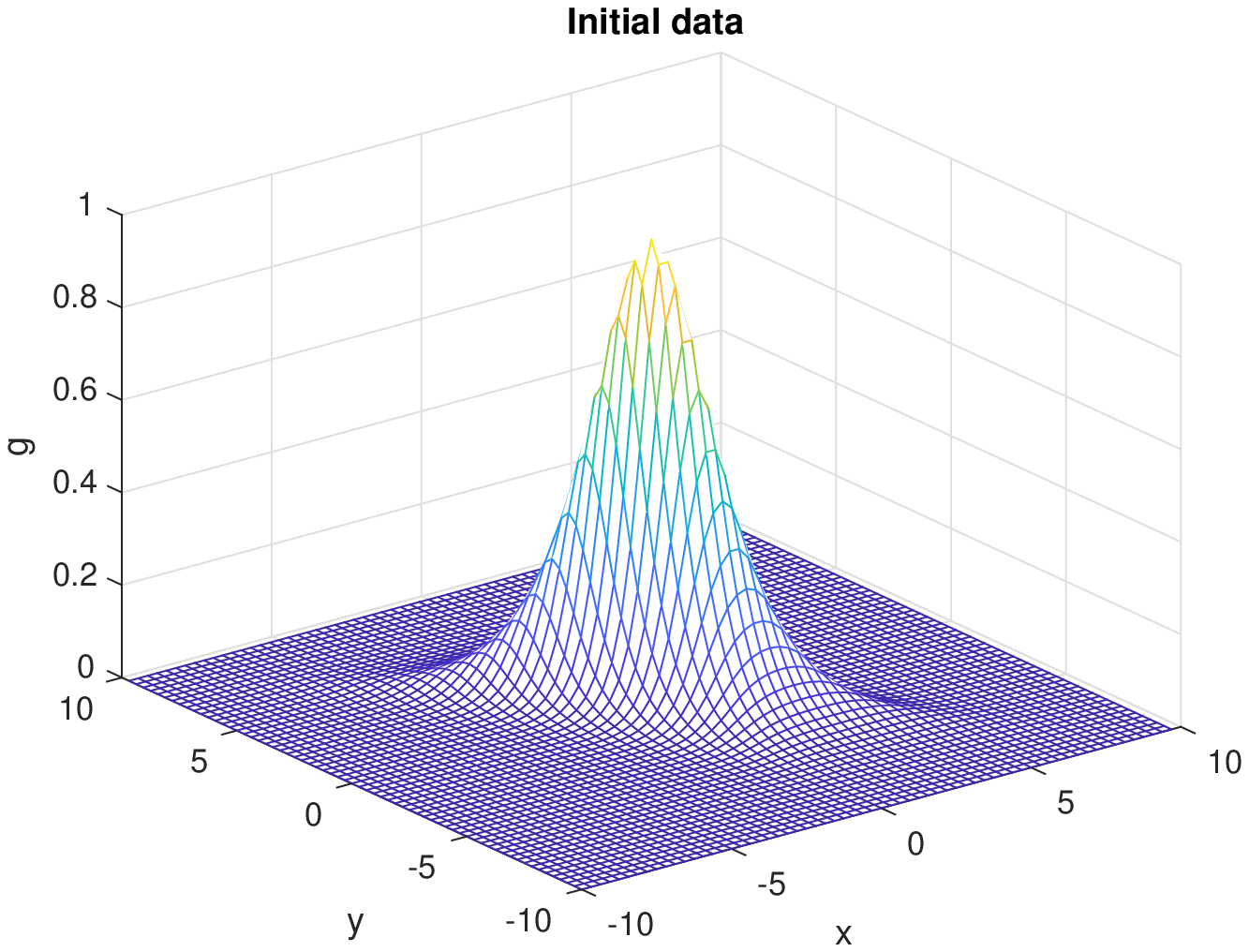}
  \includegraphics[width=6cm,height=6cm]{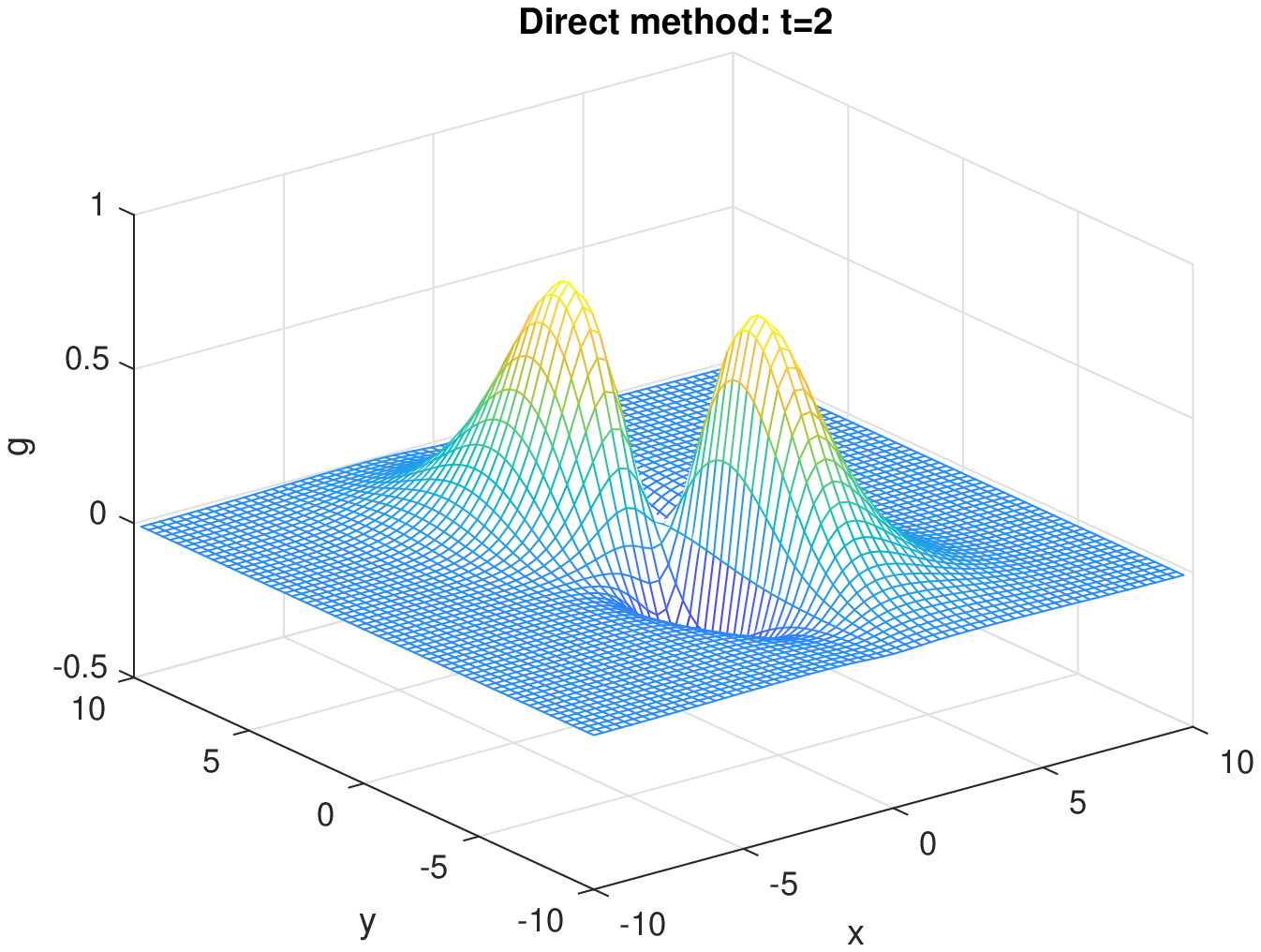}
  \includegraphics[width=6cm,height=6cm]{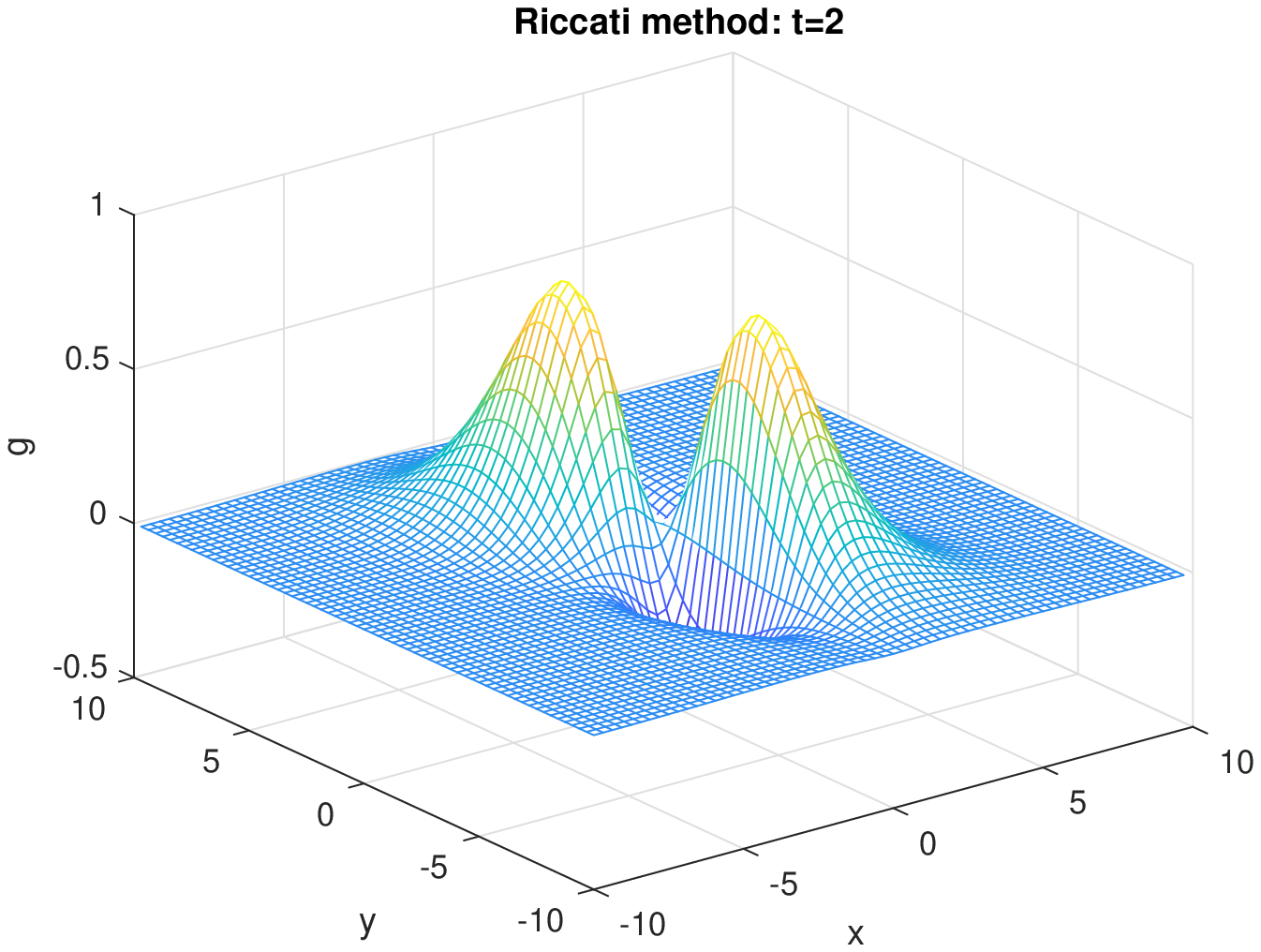}
  \end{center}
  \caption{We plot the solution to the nonlocal quadratically 
nonlinear partial differential equation with correlation from Example~\ref{ex:corr}. 
We used a generic initial profile $g_0$ as shown in the top panel. 
For time $t=2$, the middle panel shows the solution computed using 
a direct integration approach while the bottom
panel shows the solution computed using our Riccati approach.}
\label{fig:corr}
\end{figure}

\begin{example}[Nonlocal quadratic nonlinearity with correlation]\label{ex:corr}
In this case the target evolutionary partial differential equation 
has a nonlocal quadratic nonlinearity involving a correlation function 
and has the form
\begin{equation*}
\pa_tg(x,y;t)=d(\pa_x)\,g(x,y;t)-\int_{\R}g(x,z;t)\,b(z)\,g(z,y;t)\,\rd z.
\end{equation*}
This corresponds to the evolutionary partial differential equation 
with nonlocal quadratic nonlinearity in Corollary~\ref{thm:main},
with $a=c=0$ and $b=b(x)$ the scalar smooth, bounded square-integrable function
described in the paragraphs preceding it. We also assume that $d=d(\pa_x)$ 
is of the diffusive or dispersive form described in Remark~\ref{rmk:diffdisp}. 
We assume smooth and square-integrable initial data $g_0=g_0(x,y)$.

To find solutions to the evolutionary partial differential equation 
just above using our approach we assume 
the linear base and auxiliary equations have the form 
\begin{align*}
\pa_t p(x,y;t)&=d(\partial_x)\,p(x,y;t),\\
\pa_t q^\prime(x,y;t)&=b(x)\,p(x,y;t).
\end{align*}
In Fourier space the solution of the base equation has the form
\begin{equation*}
\hat p(k,y;t)=\mathrm{e}^{d(2\pi\mathrm{i}k)\,t}\,\hat p_0(k,y),
\end{equation*}
where $\hat p_0$ is the Fourier transform of the initial data for $p$.
The auxiliary equation solution in Fourier space has the form,
\begin{equation*}
\hat q^\prime(k,y;t)
=\int_\R\hat{b}(k-\kappa)\,\hat I(\kappa,t)\,\hat p_0(\kappa,y)\,\rd\kappa,
\end{equation*}
where we set
\begin{equation*}
\hat I(k,t)\coloneqq
\frac{\mathrm{e}^{d(2\pi\mathrm{i}k)\,t}-1}{d(2\pi\mathrm{i}k)}.
\end{equation*}
As in the last example we took the initial data for $q^\prime$ to be zero and thus
the initial data for $\hat q^\prime$ is also zero. We also set the 
initial data for $p=p(x,y;t)$ to be $p_0(x,y)=g_0(x,y)$.
In Fourier space this is equivalent to $\hat p_0(k,y)=\hat g_0(k,y)$.
We now derive an explicit form for $q^\prime=q^\prime(x,y;t)$ from 
$\hat q^\prime=\hat q^\prime(k,y;t)$ above. 
Taking the inverse Fourier transform of $\hat q^\prime$,
we find that
\begin{align*}
q^\prime(x,y;t)=&\;
\int_\R\biggl(\int_\R\mathrm{e}^{-2\pi\mathrm{i}k x}\,\hat{b}(k-\kappa)\,\rd k\biggr)
\hat I(\kappa,t)\,\hat g_0(\kappa,y)\,\rd\kappa\\
=&\;\int_\R\bigl(\mathrm{e}^{-2\pi\mathrm{i}\kappa x}\,b(x)\bigr)
\hat I(\kappa,t)\,\hat g_0(\kappa,y)\,\rd\kappa\\
=&\;b(x)\int_\R\mathrm{e}^{-2\pi\mathrm{i}\kappa x}\,
\hat I(\kappa,t)\,\hat g_0(\kappa,y)\,\rd\kappa\\
=&\;b(x)\int_\R I(x-\zeta,t)\, g_0(\zeta,y)\,\rd\zeta.
\end{align*}
Lastly, the Riccati relation here has the form
\begin{equation*}
p(x,y;t)=g(x,y;t)+\int_{\R}g(x,z;t)\,q^\prime(z,y;t)\,\rd z.
\end{equation*} 
Since we have an explicit expression for $q^\prime=q^\prime(x,y;t)$,
and we can obtain one for $p=p(x,y;t)$ by taking the inverse
Fourier transform of the explicit expression for $\hat p=\hat p(k,y;t)$
above, we can solve this linear Fredholm equation for $g=g(x,y;t)$.
The solution, by Corollary~\ref{thm:main}, will be the 
solution to the evolutionary partial differential equation
with the nonlocal quadratic nonlinearity above corresponding to
the initial data $g_0(x,y)$.

We solved the Fredholm equation for $g=g(x,y;t)$ numerically. 
The results are shown in Figure~\ref{fig:corr}. We set
the operator $d=\partial_x^2+1$ and took as the generic initial profile 
$g_0(x,y)\coloneqq \mathrm{sech}(x+y)\,\mathrm{sech}(y)$.
We set $b=b(x)$ to be a mean-zero Gaussian density function 
with standard deviation $0.01$. 
The top panel in Figure~\ref{fig:corr} shows the initial data.
The middle panel in the figure shows the solution profile computed 
at time $t=2$ using a direct spectral integration approach. By this
we mean we solved the equation for $\hat g(k,y;t)$ generated by taking the 
Fourier transform of the equation for $g=g(x,y;t)$. 
We used the inbuilt Matlab integrator \texttt{ode23s} to integrate
in time. The bottom panel in Figure~\ref{fig:corr} shows the solution 
computed with the time parameter $t=2$ using our Riccati approach, 
i.e.\/ by numerically solving the Fredholm equation for $g=g(x,y;t)$
above by standard methods for such integral equations. 
As expected, the solutions in the middle and bottom panels 
look identical (up to numerical precision).
\end{example}

\begin{remark} 
We emphasize that, when we can explicitly solve for 
$p=p(x,y;t)$ and $q^\prime=q^\prime(x,y;t)$ in our Riccati approach, then 
time $t$ plays the role of a parameter. One decides the time at 
which one wants to compute the solution and we then solve the Fredholm equation 
to generate the solution $g=g(x,y;t)$ for that time~$t$.
This is one of the advantages of our method over
standard numerical schemes\footnote{We quote
from the referee: ``numerical integration in time will usually become
inaccurate for large time $t$, but the nature of the exact solution gives
you a precise answer for arbitrary $t$, and maybe allows access to information
about long time behaviour which is inaccessible via standard numerical schemes.''}.
\end{remark}

\begin{remark}[Burgers' equation]\label{ex:Burgers}
Burgers' equation can be considered as a special case of our
Riccati approach in the following sense. Suppose the linear base 
and auxiliary equations are $\pa_tp(x;t)=\pa_x^2p(x;t)$ and 
$\pa_tq(x;t)=\pa_xp(x;t)$ for the real valued functions $p$
and $q$. Further suppose the Riccati relation takes the form
$p(x;t)=g(x;t)\,q(x;t)$ where $g$ is also real valued. 
Note this represents a rank one relation between $p$ and $q$
in the sense that we obtain $p$ from $q$ by a simple multiplication
of $q$ by the function $g$. From the linear base and auxiliary
equations, assuming smooth solutions, we deduce that 
$\pa_tq=\pa_xp=\pa_x^{-1}\pa_x^2p=\pa_x^{-1}\pa_tp=\pa_t(\pa_x^{-1}p)$, 
where $\pa_x^{-1}w$ represents the operation $\int_{-\infty}^xw(z)\,\rd z$
for any smooth integrable function $w=w(x)$ on $\R$. From the 
above equalities we deduce $p(x;t)=\pa_xq(x;t)+f(x)$ where $f=f(x)$
is an arbitrary function of $x$ only. If we take the special case $f\equiv0$,
then we deduce $p(x;t)=\pa_xq(x;t)$. This also implies $\pa_tq=\pa_x^2q$.
If we insert the relation $p(x;t)=\pa_xq(x;t)$ into the Riccati relation 
we find 
\begin{equation*}
g(x;t)=\frac{\pa_xq(x;t)}{q(x;t)}.
\end{equation*}
This is almost the Cole--Hopf transformation, its just missing the 
usual `-2' factor on the right-hand side. However carrying through our Riccati
approach by direct computation, differentiating the Riccati relation with
respect to time, we observe 
\begin{align*}
(\pa_tg)\,q&=\pa_tp-g\,\pa_tq\\
&=\pa_x^2p-g\,\pa_tq\\
&=\pa_x^2(g\,q)-g\,\pa_tq\\
&=(\pa_x^2g)\,q+2\,(\pa_xg)\,\pa_xq+g\,(\pa_x^2q)-g\,\pa_tq\\
&=(\pa_x^2g)\,q+2\,(\pa_xg)\,p\\
&=(\pa_x^2g)\,q+2\,(\pa_xg)\,g\,q.
\end{align*}
If we divide through by the function $q=q(x;t)$ we conclude that $g=g(x;t)$
satisfies the nonlinear partial differential equation
\begin{equation*}
\pa_tg=\pa_x^2g+2\,g\,\pa_xg.
\end{equation*}
However we now observe that `$-2\,g$' indeed satisfies Burgers' equation.
\end{remark}

\section{Conclusions}\label{sec:conclu}
There are many extensions of our approach to more general
nonlinear partial differential equations. One immediate extension
to consider is to multi-dimensions, i.e.\/ where the underlying
spatial domain lies in $\R^n$ for some $n\geqslant1$. This 
should be straightforward as indicated in Remark~\ref{rmk:multi-d}  
above. Another immediate extension is to systems of 
nonlinear partial differential equations with nonlocal nonlinearites. 
Indeed we explicitly consider this extension 
in Beck, Doikou, Malham and Stylianidis~\cite{BDMStrans} where we
demonstrate how to generate solutions to certain classes of
reaction-diffusion systems with nonlocal quadratic nonlinearities.
We also demonstrate therein, how to extend our approach to 
generate solutions to evolutionary partial differential equations
with higher degree nonlocal nonlinearities, including the nonlocal
nonlinear Schr\"odinger equation. Further therein, 
for arbitrary initial data $g_0=g_0(x)$, 
we use our Riccati approach to generate solutions to the 
nonlocal Fisher--Kolmogorov--Petrovskii--Piskunov equation
for scalar $g=g(x;t)$ of the form 
\begin{equation*}
\pa_tg(x;t)=d(\pa_x)\,g(x;t)-g(x;t)\int_\R g(z;t)\,\rd z.
\end{equation*}
This has recently received some attention; see Britton~\cite{Britton}
and Bian, Chen and Latos~\cite{BCL}. We would also like to consider
the extension of our approach to the full range of possible choices
of the operators $d$ and $b$ both as unbounded and bounded operators,
for example to fractional and nonlocal diffusion cases. We have already
considered the extension of our approach to evolutionary
stochastic partial differential equations with nonlocal nonlinearities
in Doikou, Malham and Wiese~\cite{DMW}. Therein we consider the separate
cases when the driving space-time Wiener field appears as a nonhomogeneous
additive source term or as a multiplicative but linear source term. 
Of course, another natural extension is to determine whether we
can include the generation of solutions to evolutionary partial
differential equations with local nonlinearities within the context
of our Riccati approach. One potential approach is to suppose
the Riccati relation is of Volterra type. This is an ongoing investigation. 
Lastly we remark that for the classes of 
nonlinear partial differential equations we can consider, 
solution singularities correspond to poor choices of coordinate patches
which are related to function space regularity. In principle solutions 
can be continued by changing coordinate patches; 
see Schiff and Shnider~\cite{SchiffShnider} 
and Ledoux \textit{et al.\/}~\cite{LMT}. This is achieved by 
pulling back the flow to the relevant general linear group and then
projecting down to a more appropriate coordinate patch of the Fredholm
Grassmannian. Alternatively, we could continue the flow in the 
appropriate general linear group via the base and auxiliary equations, 
and then monitor the relevant projection(s).

% Auxiliary equation can be nonlinear (in the case a=0 and c=0).
% Need to include a section on factorization
% Wavelets
% Supersymmetry connection

\begin{acknowledgement}
We are very grateful to the referee for their detailed report 
and suggestions that helped significantly improve the original manuscript.
We would like to thank Percy Deift, Kurusch Ebrahimi--Fard and 
Anke Wiese for their extremely helpful comments and suggestions.
The work of M.B. was partially supported by 
US National Science Foundation grant DMS-1411460.
\end{acknowledgement}

\end{document}